\documentclass[journal]{IEEEtran}
\IEEEoverridecommandlockouts
\usepackage[noadjust]{cite}
\usepackage{amsmath,amssymb,amsfonts}
\usepackage{amsthm}
\usepackage{algorithmic}
\usepackage{relsize}
\usepackage{algorithm}
\usepackage{graphicx}
\usepackage{subcaption}
\usepackage{comment}
\usepackage{textcomp}
\usepackage{xcolor}
\usepackage{url}
\usepackage{xcolor}
\usepackage{tikz}
\def\BibTeX{{\rm B\kern-.05em{\sc i\kern-.025em b}\kern-.08em
		T\kern-.1667em\lower.7ex\hbox{E}\kern-.125emX}}
\usepackage{enumerate}
\usepackage{amsbsy}

\usepackage{mathtools}
\usepackage{placeins}

\allowdisplaybreaks
\begin{document}
	
	\title{Robust Localization with Bounded Noise: Creating a Superset of the Possible Target Positions via Linear-Fractional
Representations}
	\author{Jo\~ao Domingos,  Cl\'audia Soares and  Jo\~ao Xavier}
	\newtheorem{theorem}{Theorem}
	\newtheorem{lemma}{Lemma}
	\newtheorem{remark}{Remark}
	\newtheorem{example}{Example}
	\newtheorem{result}{Result}
	\newtheorem{corollary}{Corollary}[theorem]
	\newtheorem{definition}{Definition}
	\maketitle
	
\begin{abstract}
Locating \textcolor{black}{a target} is key in many applications, namely in high-stakes real-world scenarios, like detecting humans or obstacles in vehicular networks. \textcolor{black}{In scenarios where precise statistics of the measurement noise are unavailable, applications require localization methods that assume minimal knowledge on the noise distribution. We present a scalable algorithm delimiting a tight superset of all possible target locations, assuming range measurements to known landmarks, contaminated with bounded noise and unknown distributions. This superset is of primary interest in robust statistics since it is a tight majorizer of the set of Maximum-Likelihood (ML) estimates parametrized by noise densities respecting two main assumptions: (1) the noise distribution is supported on a ellipsoidal uncertainty region and (2) the measurements are non-negative with probability one. } We \textcolor{black}{create} the superset through convex relaxations that use Linear Fractional Representations (LFRs), a well-known technique in robust control.
 \textcolor{black}{ For low noise regimes the supersets created by our method double the accuracy of a standard semidefinite relaxation. For moderate to high noise regimes our method still improves the benchmark but the benefit tends to be less significant, as both supersets tend to have the same size (area). }
\end{abstract}
	
	
	\section{Introduction}
	\label{sec:introduction}
	
\textcolor{black}{\IEEEPARstart{T}{he} problem of locating a unknown target from noisy range measurements is a long-established problem in the signal processing community. A popular approach is the Maximum-Likelihood (ML) estimator which, given an assumed noise distribution $f_{\Delta}$, computes a point estimate $\hat{x}$ of the target position. The downfall of relying on a single point estimate $\hat{x}$  is that, for a fixed set of measurements, the value of $\hat{x}$ is dependent on the noise distribution $f_\Delta$. In concrete, different choices of $f_\Delta$ will yield different estimates $\hat{x}$ -- see figure~\ref{fig:motivating_plot}. In this paper we assume that a precise noise density $f_{\Delta}$ is unavailable and, to overcome this challenge, we propose an algorithm that creates a superset of all possible target estimates $\hat{x}$ consistent with two minimal assumptions on the problem: (1) the noise is bounded and (2) the measurements are always non-negative.
\\~\\
In terms of figure~\ref{fig:motivating_plot} we want to create a superset  of all the blue points that arise from distributions $f_{\Delta}$ respecting the aforementioned minimal assumptions. Our algorithm is particularly useful for applications that require minimal knowledge on the noise distribution $f_{\Delta}$ such as search and rescue operations. For example assume that a plane is lost in sea, and we have access to range measurements to known landmarks. By creating a superset of all estimates $\hat{x}$ we can design a simple rescue area (a rectangle, for example) that contains all positions where the plane might have crashed. In this case, specifying a single position $\hat{x}$ for the plain  might lead to a wrong location thus, leading to the failure of the operation. Specifying a region where the plain must necessarily lie enables an efficient successful search and rescue operation.}
\begin{figure}[tb]
\centering
	\includegraphics[width=6cm]{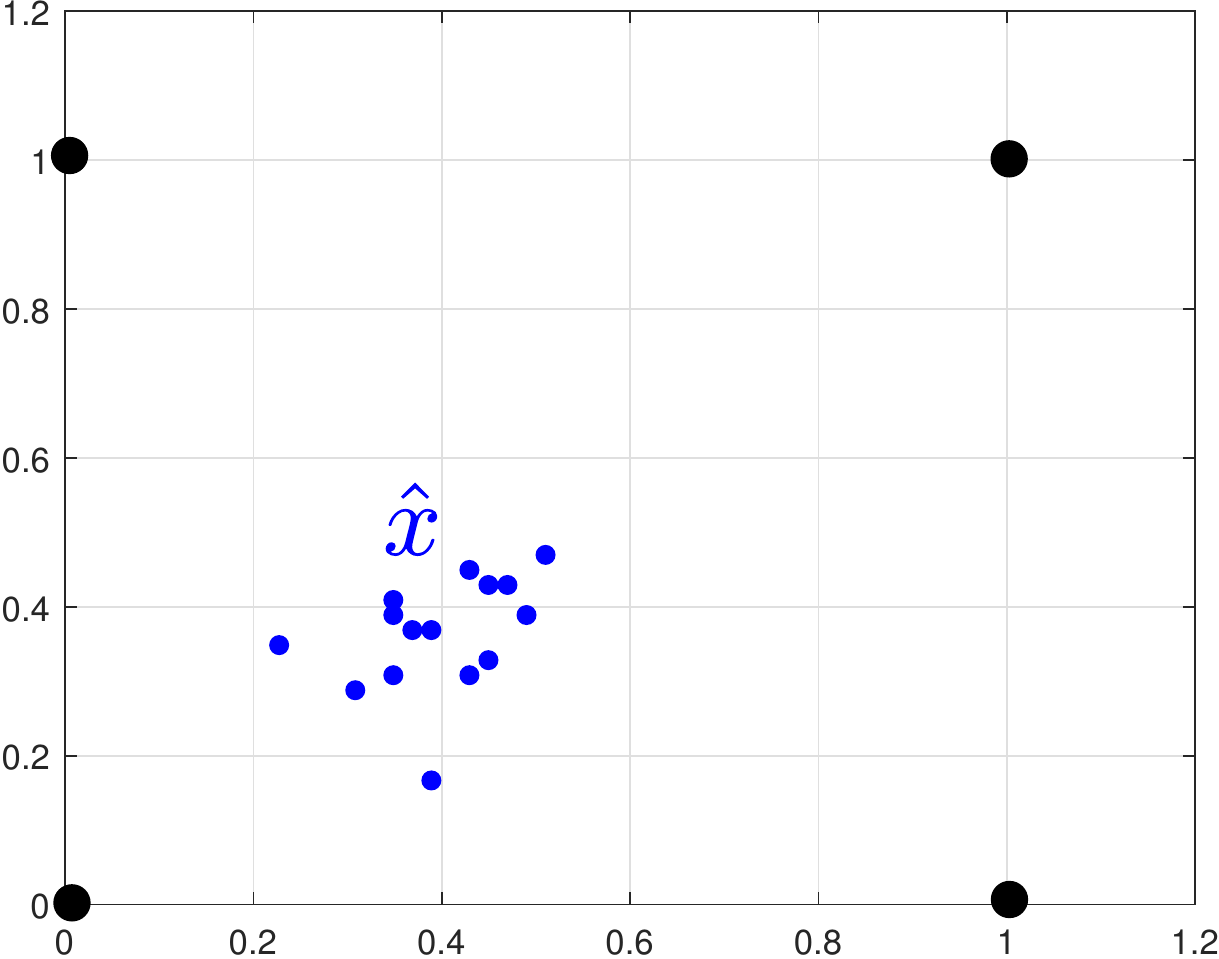}
	\caption{ \textcolor{black}{ Different ML point estimates $\hat{x}$ for different noise densities $f_{\Delta}$. The black dots represent reference landmarks in a square region with 1Km side. Landmarks measure the distance from themselves to a unknown target. We assume that measurements are affected by additive noise with some density $f_{\Delta}$. The blue dots represent ML estimates $\hat{x}$ for different noise densities $f_{\Delta}$ with a baseline uncertainty (standard deviation) of $0.1$ Km ($100$ meters). We consider $15$ Gaussian densities $f_{\Delta}=\mathcal{N}(0,D)$ with distinct diagonal covariances $D$. The noise densities $f_{\Delta}$ vary by assuming that zero, one, two or three landmarks have precise sensors with an uncertainty of $0.01$ Km ($10$ meters). The ML estimates $\hat{x}$ are computed by a  grid search over $[-0.5,1.5]^2$. Different choices of $f_{\Delta}$ lead to distinct estimates $\hat{x}$: if we are insure about the uncertainty of each sensor which estimate $\hat{x}$ should we choose?} }
	\label{fig:motivating_plot}
\end{figure}
\subsection{Literature Review}
\label{sec:related-wk}

\paragraph*{Pointwise range-based target localization}
The literature around range-based target localization is vast when considering point estimates of the position of a target from range measurements to precisely known locations, termed anchors. The seminal paper by Beck~\cite{beck2008exact}, was followed by many, notably~\cite{yang2009efficient}, which also models process noise as i.i.d.\ Gaussian random variables. The work by Oguz-Ekim et al.~\cite{ouguz2010convex} considers outlier measurements, by assuming  measurement errors with Laplace distribution. Other relevant literature on range-based target localization is~\cite{5714759,6731596,7747483,8521706,7926459,8417931,9089013,gao2020majorization}.

\paragraph*{Robust network localization} The problem of robust network localization, where there are many unlocalized nodes with access to a few noisy pairwise distance measurements affected by outliers, has been addressed by a few works. One of them relies on identifying outliers from regular data and discarding them, as in Ihler et al.~\cite{IhlerFisherMosesWillsky2005} that formulates the problem using a probabilistic graphical model to encode the data distribution. 
Vaghefi et al.~\cite{vaghefi2015cooperative} proposed a semidefinite relaxation of a model considering unknown, unbounded outlier errors for the cooperative localization scenario. Forero el al.~\cite{ForeroGiannakis2012} presented a robust multidimensional scaling based on regularized least squares, where the robust regularization term was relaxed to a convex function. Korkmaz and van der Veen~\cite{KorkmazVeen2009} use the Huber loss composed with a discrepancy function between measurements and estimated distances, in order to achieve robustness to outliers. Recently, Soares and Gomes~\cite{soares2021strong} have proposed a distributed Huber-based point estimator for range measurements corrupted with Gaussian measurement noise, and non-Gaussian, long-tail noise, modelled as Laplace or Cauchy noise. The work in~\cite{chakraborty2020cooperative} uses range-only data to initialize an extended Kalman Filter for sensor fusion data.

\paragraph*{Robust range-based localization}
The space of range-based \emph{robust target localization} was explored by several authors, for example the work in~\cite{Sun2004}, where a robust M-estimator was applied to target localization in a bootstrapping scheme in~\cite{YinZoubirFritscheGustafsson2013}, motivated by outlier measurements generated by non-line of sight (NLOS) propagation of signals, where bouncing on obstacles causes large delays in the time of flight, and thus a large error in the estimated distance. 
Bootstrapping with the Huber M-estimator was already used for centralized target localization in [6], where the Huber estimation was used in a bootstrapping scheme.
Wang and colleagues~\cite{wang2014nlos} address the target localization problem by relaxing the NLOS-aware problem to both a semidefinite and a second-order cone problem.
Later, Tomic et al.~\cite{tomic2017robust} modeled  an additive Gaussian noise term, plus a bounded NLOS bias as a nuissance parameter to be jointly estimated, and formalized the problem as a trust region subproblem, solved by bisection. Still, the model considers an infinite support for the errors and the solution is a pointwise estimate of the target position.
The recent work of Chen and colleagues~\cite{Chen2019} puts forward a new model for LOS/NLOS based on a multiplicative transformation of the additive data model, considering Exponential noise. The authors argue that this type of noise is routinely found in dense urban areas.
Other relevant works include~\cite{7426865,9257378,guvenc2009survey,prorok2012online,cong2005nonline,9350568,chen2019improved,marano2010nlos}.

\paragraph*{Target localization with bounded noise}
Another line of research assumes unknown and bounded noise, which, in practice, could be a more reasonable model, considering that setup and hardware specifications are, in general, known. The authors in~\cite{Shi2016} consider bounded errors with unknown distribution in range measurements, and compute a point-wise estimate by minimizing the worst-case position estimation error.

\paragraph*{Delimiting the set of all possible solutions for an estimation problem considering bounded noise}
A few papers examine the important problem of, given a data model, determining the region where all possible estimates compatible with observed data may lie. Eldar et al.~\cite{eldar2008minimax} develop a convex solution using Lagrange duality to a data model with linear dependency relative to the unknown parameter and added Gaussian noise.
\vspace{-0.1cm}
\subsection{\textcolor{black}{Problem Statement}} We consider the problem of delimiting the region of possible positions of a target, given noisy range measurements from known landmarks. Denoting the position of the target by $x \in {\mathbf R}^d$ (in practice, $d \in \{  2, 3\}$) and the position of the landmarks by $r_m \in {\mathbf R}^d$, $1 \leq m \leq M$, we have the model 
\begin{equation} y_m = \left\| x - r_m \right\| + u_m. \label{eq:model} \end{equation}
Here, $y_m \in {\mathbf R}$ is the $m$th available measurement and $u_m \in {\mathbf R}$ represents unknown additive noise. The symbol $\left\| \cdot \right\|$ denotes the  Euclidean norm. To simplify notation, in~\eqref{eq:model} and throughout the paper  each constraint involving $m$ is to be understood as a set of $M$ constraints, one per $m$ in the set $\{ 1, \ldots, M \}$. 

\paragraph*{Assumptions}
We make two assumptions, one on the available measurement vector $y = ( y_1, \ldots, y_M ) \in {\mathbf R}^M$ and one on the unknown noise vector $u = ( u_1, \ldots, u_M) \in {\mathbf R}^M$:
\begin{itemize}
	\label{assumptions}
	\item \emph{Assumption 1 ($y$ is non-negative):} the measurement vector given by~\eqref{eq:model} is nonnegative, that is, $y_m \geq 0$ for each $m$.
	\item \emph{Assumption 2 ($u$ is bounded):} the noise vector lies in a known ellipsoid; specifically, $u$ is a member of an ellipsoid $\mathcal{E}(0,\Sigma)$ centered at the origin and described by  \begin{align} \mathcal{E}(0,\Sigma) = \{ v \in {\mathbf R}^M \colon v^T \Sigma^{-1} v \leq 1 \}, \label{eq:ellipsoid} \end{align} 
	where $\Sigma$ is a known $M\times M$ positive-definite  matrix.
\end{itemize}
Both assumptions are mild. Assumption 1 is natural because each $y_m$ is a range measurement. As such, $y_m$ represents a physical distance, which can only be non-negative. Assumption 2 essentially says that the unknown noise vector has a bounded support, which is also a mild assumption: in general, hardware characteristics and the physical setup will naturally limit how large noisy measurements can get. If the support were unbounded, components of $u$ could get negative enough to create the paradox of negative measurements in the data model~\eqref{eq:model}, for any fixed target position. Finally,  Assumption 2  does not impose any particular noise distribution (we return to Assumption 2 in more detail in Section~\ref{sec:stats}).

\paragraph*{The set of possible target positions} The available measurement vector  $y$ can typically be resolved into infinitely many pairs $(x, u) \in {\mathbf R}^d \times \mathcal{E}(0,\Sigma)$ that satisfy~\eqref{eq:model}. Our main interest is in the set $\mathcal X$ of all possible target positions $x$, denoted \begin{equation} {\mathcal X} = {\mathcal X}_1 \cap {\mathcal X}_2, \label{eq:defX} \end{equation}
where ${\mathcal X}_1$ is the set of target positions that can explain the measurement vector $y$, and ${\mathcal X}_2$ is the set of target positions that always lead to nonnegative measurements. Formally, 
\begin{align} {\mathcal X}_1 &= \left\{ \right. x \in {\mathbf R}^d \colon \, \exists \,\,  u \in {\mathcal E}(0,\Sigma),\,\,  y_m = \left\| x - r_m \right\| + u_m \left. \right\} \label{eq:defX1} \\
{\mathcal X}_2 &= \left\{ \right. x \in {\mathbf R}^d \colon \, \forall \,\,  u \in {\mathcal E}(0,\Sigma),\,\,  \left\| x - r_m \right\| + u_m \geq 0 \left. \right\} \label{eq:defX2} \end{align}
The set ${\mathcal X}_1$ depends on the available measurement $y$ and describes the set of target positions that can explain the $y$ at hand. The set ${\mathcal X}_2$ does not depend on $y$ and is included to guarantee internal consistency between the data model and Assumption 1 about the non-negativity of range measurements. Specifically, the set ${\mathcal X}_2$ retains the target positions that always generate nonnegative measurements.
By using definition~\eqref{eq:ellipsoid} we can rewrite set ${\mathcal X}_2$ as   \begin{equation} {\mathcal X}_2 = \left\{ x \in {\mathbf R}^d \colon \left\| x - r_m \right\|^2 \geq \Sigma_{mm} \right\}, \label{eq:defX2} \end{equation} where $\Sigma_{mm}$ is the $m$th diagonal entry of the matrix $\Sigma$ in~\eqref{eq:ellipsoid}. \textcolor{black}{ Figure~\ref{fig:set_X_M_1} represents set $\mathcal{X}$ in the simple scenario where a single landmark tries to locate the target. Note that $\mathcal{X}$ is already non-convex. Furthermore our numerical experiments (Section~\ref{sec:numerical_results}) consider a setup with $M=3$ anchors. In this case ${\mathcal X}$ can have an unwieldy shape: it can be nonconvex with sharp extreme points or even disconnected (see Figure~\ref{fig:worst_average_best} in\footnote{In Figure~\ref{fig:worst_average_best}  we are computing a grid approximation of $\mathcal{X}$ so, in full rigour, we can only describe $\mathcal{X}$ up to the resolution of the mentioned grid. More details in section~\ref{sec:numerical_results}.} Section~\ref{sec:numerical_results}).}
\begin{figure}[tb]
	\hspace{0.3cm}
	\includegraphics[width=8.1cm]{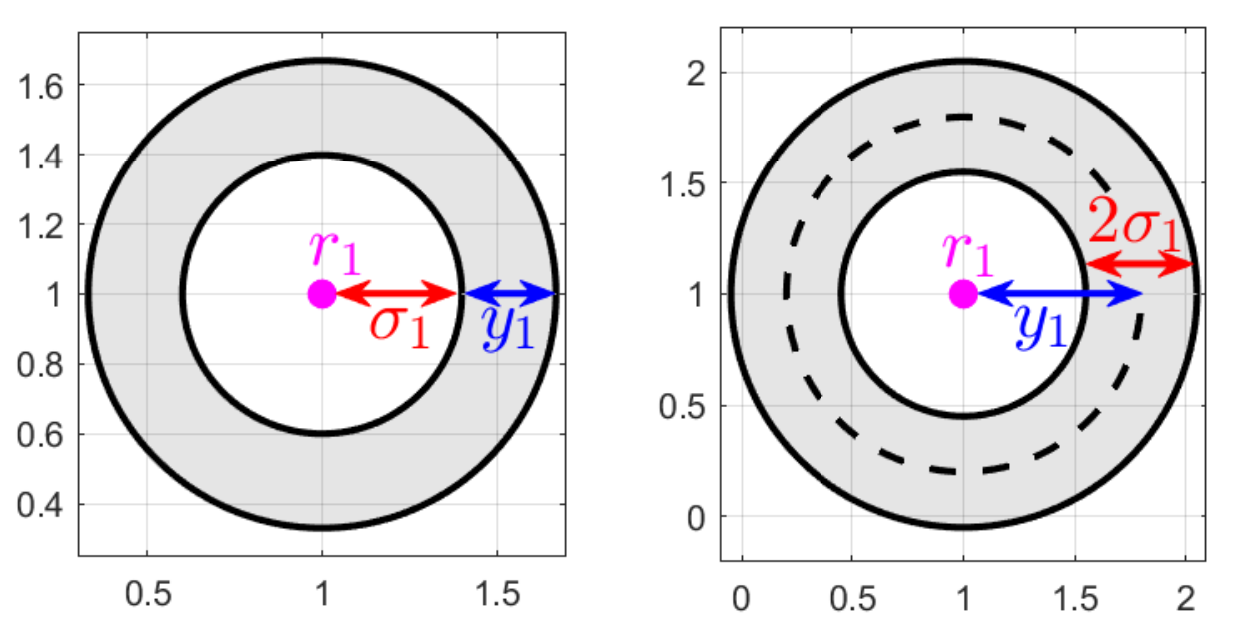}
	\caption{ \textcolor{black}{Set $\mathcal{X}$ (in gray) for $M=1$, that is, when we have only one non-negative measurement $y_1\geq 0$, one reference landmark $r_1\in \mathbf{R}^d$ and one uncertainty level $\sigma_1:=\sqrt{\Sigma_{1,1}}$. The left figure plots $\mathcal{X}$ when $y_1$ is small, i.e., $y_1\in[0,2\sigma_1]$. The right figure plots $\mathcal{X}$ when $y_1$ is large enough, i.e., $y_1>2\sigma_1$. When $y_1$ is large enough, we get that $\mathcal{X}=\mathcal{X}_1$, that is, the non-negativeness assumption is already expressed in $\mathcal{X}_1$.} }
	\label{fig:set_X_M_1}
\end{figure}
\\~\\
Our goal is to compute an outer approximation of ${\mathcal X}$, which we denote  by $\overline {\mathcal X}$. Because ${\mathcal X} \subseteq \overline {\mathcal X}$, the superset $\overline {\mathcal X}$ delimits all possible target positions that explain the given measurements, under the two assumptions on the data model. Importantly,  the approximation should be   tractable to compute and as tight as possible. We let $\overline {\mathcal X}$ have the shape of a rectangle, \begin{equation} \overline {\mathcal X} = \left\{ x \in {\mathbf R}^d \colon\underline{\beta}   \leq x \leq \overline{\beta} \right\}, \label{eq:rect} \end{equation} where $\underline{\beta}, \overline{\beta} \in {\mathbf R}^d$ are vectors to be determined. The inequalities inside the braces mean that $\underline{\beta}_i \leq x_i \leq \overline{\beta}_i$ for $1 \leq i \leq d$, with $\underline{\beta}_i$, $x_i$, and $\overline{\beta}_i$ being the $i$th component of $\underline{\beta}$, $x$, and $\overline{\beta}$, respectively. The rectangular form is adopted for simplicity only; the proposed approach extends readily to any other desired form of polyhedron -- \textcolor{black}{see Section~\ref{sec:extensions}}. A rectangle is interesting in practice because it allows to delimit a useful perimeter of possible target positions, say, for search-and-rescue operations.
\\~\\
\textcolor{black}{ The work of Eldar et al.~\cite{eldar2008minimax} also considers the problem of approximating the set of all estimates given bounded error, under some assumptions about the parameter/vector to be estimated. There are two main difference between our work and~\cite{eldar2008minimax}. First, our observation model in~\eqref{eq:model} is non-linear, while~\cite{eldar2008minimax} only considers linear models. Second,  we approximate the set of estimates by a rectangle or, more generally, a polyhedron -- see Section~\ref{sec:extensions}. Eldar et al.~\cite{eldar2008minimax} compute a ball with small radius by approximating the Chebyshev center of the set of all estimates. The non-linear model~\eqref{eq:model} prevents us from using the results of~\cite{eldar2008minimax} since, in our case, the set of all estimates $\mathcal{X}$ is generally  non-convex -- see figure~\ref{fig:worst_average_best}. Eldar et al.~\cite{eldar2008minimax} consider a convex set of estimates given by a finite intersection of ellipsoids.}
\subsection{Contributions}
\label{sec:contributions}
	A preliminary version of this work also used linear fractional representations to approach the problem of target localization~\cite{first_LFR_paper}. This work expands comprehensively the early version along four main directions:
	\begin{enumerate}
		\item Scalability: In~\cite{first_LFR_paper} our LFR modelling leads to a flattening map (explained ahead) $L_{{\mathcal U}_{[2M+M^2]}}$ that has an input dimension
		 $\mathcal{O}(M^3)$ and output dimension $\mathcal{O}(M^2)$,  with $M$ denoting the number of anchors. In this paper we achieve a
		flattening map $L_{{\mathcal U}_{[2M]}}$ with both lower input $\mathcal{O}(M^2)$ and output $\mathcal{O}(M)$ dimensions. 
		This decrease in dimensions is key for scalability  since, for a fixed dimension $d$, our approach solves $2d$ semidefinite programs each with $\mathcal{O}(M^2)$  variables (input dimension of $L_{{\mathcal U}_{[2M]}}$) and a linear matrix inequality (LMI) in\footnote{For an arbitrary $d$, $\mathbf{S}^{d}$ denotes the set of $d \times d$ symmetric matrices. } $\mathbf{S}^{2M}$ (output dimension of $L_{{\mathcal U}_{[2M]}}$). The modeling approach of~\cite{first_LFR_paper} leads to larger semidefinite programs with $\mathcal{O}(M^3)$ variables and an LMI of dimension $\mathcal{O}(M^2)$.
		This effective decrease in complexity is achieved by a careful manipulation of LFR modelling tools (detailed in Appendix~\ref{ap:step1}).
			\item Non-negativeness assumption: Unlike~\cite{first_LFR_paper}, in this paper we introduce set $\mathcal{X}_2$ which explicitly copes with the assumption that the measurements generated according to the data model~\eqref{eq:model} are non-negative.
		\item \textcolor{black}{Robustness interpretation: We give a robustness interpretation of the problem in terms of a frequentist formalism. Under mild assumptions, we show that the set of target positions $\mathcal{X}$ contains all Maximum Likelihood (ML) estimates of the target position $x^*$ regardless of any feasible noise density $f_{\Delta}$. Our main assumptions on $f_{\Delta}$ mirror the assumptions described in section~\ref{sec:introduction}:  $f_{\Delta}$  is supported on the ellipsoid $\mathcal{E}(0,\Sigma)$ and $f_{\Delta}$ yields non-negative measurements (almost surely). So, in simple terms, $\mathcal{X}$ majorizes the set of point estimates that are \textit{plausible} under a frequentist formalism and that respect the assumptions of the problem. Furthermore, this majorization is actually tight: any point estimate $\hat{x}\in \mathcal{X}$ is a \textit{plausible} estimate of the $x^*$, in the sense that there exists a noise density $f_{\Delta}$ such that $\hat{x}$ is an ML estimate under $f_{\Delta}$ and $f_{\Delta}$ respects the assumptions of the problem.}
		\item  \textcolor{black}{Numerical validation: We give extensive numerical experiments to validate our approach. Namely, we compare our method with a benchmark convex relaxation and verify that our  method tends to outperform the benchmark method.
		In average, our localization approach doubles the accuracy the benchmark method   when the amount of measurement noise is low.}
	\end{enumerate}
\subsection{Paper Organization}
 In section~\ref{sec:probstatement} we formalize the problem of computing outer approximations of the set $\mathcal{X}$ of possible target positions.  Section~\ref{sec:backgroundLFR} provides some background on Linear Fractional Representations (LFRs). Section~\ref{sec:our_approach} details our approach for computing an outer approximation of $\mathcal{X}$ via LFRs. In section~\ref{sec:benchmark} we outline a benchmark convex relaxation. Section~\ref{sec:stats} provides a statistical interpretation of set $\mathcal{X}$ in terms of a Frequentist formalism when the noise densities are unknown. \textcolor{black}{Section~\ref{sec:numerical_results} provides numerical evidence that, on average, our method outperforms a standard benchmark relaxation. Section~\ref{sec:extensions} shows how to extend the proposed approach to compute a general polyhedral approximation of $\mathcal{X}$. } Section~\ref{sec:conclusion} concludes the paper.
\section{The core Problem}
\label{sec:probstatement}

Our goal is to compute the components of the vectors $\underline{\beta}, \overline{\beta} \in {\mathbf R}^d$ that define the corners of the outer rectangle $ \overline {\mathcal X}$, see~\eqref{eq:rect}.

Consider the first component of $ \overline{\beta}$, component $ \overline{\beta}_1$. The best (tigthest) $ \overline{\beta}_1$ is $\sup\{ s^T x \colon x \in {\mathcal X} \}$, where $s = (1, 0, \ldots, 0 )\in \mathbf{R}^d$, because the supremum reveals how much ${\mathcal X}$ stretches in direction $(1, 0, \ldots, 0)$. This value is given by
\begin{equation} \begin{array}[t]{ll} \underset{x}{\text{maximize}} & s^T x \\ \text{subject to} & x \in {\mathcal X}.
\end{array} 
\label{eq:core} \end{equation}
Computing the remaining components of $\overline{\beta}$ amounts to solving~\eqref{eq:core} again, just with a different choice of $s$ for each component. In general, computing the best (tightest) $ \overline{\beta}_i$ is done by taking $s = (0, \ldots, 1, \ldots, 0)$ (all components of $s\in \mathbf{R}^d$ are equal to $0$ except for the $i$-th component which is equal to $1$).

Computing $ \underline{\beta}$ amounts to solving problems of the form~\eqref{eq:core}, too. The best  $ \underline{\beta}_i$ is found by first computing the maximum of~\eqref{eq:core} with $s = -(0, \ldots, 1, \ldots, 0)$ and then flipping  its sign. To conclude, we need to solve $2d$ problems of the form~\eqref{eq:core}. 

We now focus on problem~\eqref{eq:core}, for an arbitrary  $s \in {\mathbf R}^d$. 
Thus, our core problem  is to compute the optimal value of 
\begin{equation} \begin{array}[t]{ll} \underset{x, u}{\text{maximize}} & s^T x \\ \text{subject to} & y_m = \left\| x - r_m \right\| + u_m \\
& \left\| x - r_m \right\|^2 \geq \Sigma_{mm} \\ & u^T \Sigma^{-1} u \leq 1.
\end{array} 
\label{eq:core1}
\end{equation}
Because this problem is nonconvex, we propose a tractable, convex relaxation. We compute an upper-bound on the optimal value of~\eqref{eq:core1}, thereby still obtaining  a valid outer rectangle $\overline {\mathcal X}$ for ${\mathcal X}$, but maybe not the tightest one. To generate the convex relaxation, we use a technique from robust control known as linear-fractional representations (for example see ~\cite{packard1993control},~\cite{el1996control},~\cite{zhou1998essentials},~\cite{calafiore2004ellipsoidal}).

\section{\textcolor{black}{Robustness Interpretation of Set $\mathcal{X}$}}
\label{sec:stats}
\textcolor{black}{This section builds on the motivation of section~\ref{sec:introduction} and figure~\ref{fig:motivating_plot} by giving a formal statistical interpretation of $\mathcal{X}$. In loose terms, this section shows that  $\mathcal{X}$ represents the set of estimates that are optimal, under a Maximum-likelihood (ML) criteria,for noise distributions respecting the assumptions of the problem. So, any ML estimate $\hat{x}$ lies in $\mathcal{X}$, and each point $\hat{x}\in \mathcal{X}$ is optimal under the ML criteria for some feasible noise distribution. A consequence of this interpretation is that, essentially, $\mathcal{X}$ jointly considers \textit{all} statistical estimates of the target $x$. By delimiting $\mathcal{X}$ we are tracking the set of all estimates $\hat{x}$ that can arise from a ML criteria, regardless of the underlying noise distribution respecting the assumptions of the problem.}
\subsection{\textcolor{black}{Maximum-Likelihood Estimation}}
\textcolor{black}{Let us consider the set of possible target positions}
\begin{align}
\mathcal{X}:=\{ x:\enspace y-\theta(x;R)\in \Phi,\,\, \theta(x;R)+\Phi\subseteq \mathbf{R}_+^M   \} 
\label{eqn:non_convex_set}
\end{align}
(assumed non-empty) with $\Phi\subseteq \mathbf{R}^M$ an arbitrary uncertainty region for the measurements $y$ and $\theta(x;R)$ the mapping that concatenates the distance from target $x$ to each anchor $r_m$,
\begin{align}
\theta(x;R):= \{ ||x-r_m||  \}_{m=1}^M,\enspace R=(r_1,\dots,r_M).
\label{eq:true_measurements}
\end{align}
Definition~\eqref{eq:defX} works with an ellipsoidal uncertainty region $\Phi = \mathcal{E}(0,\Sigma)$ but the forthcoming interpretation holds for general closed, convex uncertainty sets with a non-empty interior\footnote{The interior of a set $\Phi\subseteq \mathbf{R}^M$, denoted as $\text{int } \Phi$, is defined as the set of interior points so $\text{int }  \Phi=\{x:  \exists\,\, \epsilon>0,\enspace B(x,\epsilon)\subseteq \Phi\}$ with $ B(x,\epsilon)$ a open ball in $\mathbf{R}^M$ so $B(x,\epsilon)=\{v: ||x-v||<\epsilon\}$.} $\text{int } \Phi$. The interpretation follows the framework of robust estimation, where $x$ is to be estimated with the underlying noise distribution unknown. Our interpretation shows that, under mild assumptions, set $\mathcal{X}$ is equal to the set of {reasonable} estimates of target position $x$ from measurements $y$, given that the measurements come from a \textit{coherent} probabilistic model defined over the uncertainty region $\Phi\subseteq \mathbf{R}^M$ but with an arbitrary noise density.
\\~\\
Consider a frequentist estimation problem where we want to estimate the position $x$ of a target from measurements
\begin{align}
Y(x)=\theta(x;R)+\Delta,
\label{eqn:freq_model}
\end{align}
where $\Delta$ is a random noise vector in $\mathbf{R}^M$ that has density $f_{\Delta}$ and support\footnote{The support~\cite{billingsley1995probability} (page 181) of a random vector $X$ in  $\mathbf{R}^p$, denoted as $S_X$, is defined as the minimal closed set that supports $X$. So, if $C$ is an arbitrary closed set, then $C$ supports $X$ (that is $\mathbb{P}(X\in C)=1$) if and only if $S_X\subseteq C$. The symbol $\mathbb{P}$ denotes a probability measure~\cite{billingsley1995probability}.} $S_\Delta$. The notation $Y(x)$ indicates that the measurements depend on the position $x$ of the target.
Given this setup, a {popular} estimator is the Maximum-Likelihood (ML) estimator
\begin{align}
\hat{x} \in \arg \max_x f_{\Delta}(y-\theta(x;R)).
\label{eqn:ML_estimation}
\end{align}
The ML estimate however has the drawback that is may be {unrealistic} in the sense that it could render negative observations with non zero probability. So even if the support set $\Phi$ is bounded we can have $\mathbb{P}(\exists_m :  Y_m(\hat{x})< 0)>0,$ with $Y_m(\hat{x})$ the $m$-th component of the observation vector resulting from~\eqref{eqn:freq_model} with $x=\hat{x}$. This event is {unrealistic} because we can observe only non-negative measurements; so our estimate should display this property we know to be true. To highlight this feature we present the notion of \textit{coherent} estimates $\hat{x}$.
\begin{definition}
	\label{def:coherent}
	An arbitrary vector $\hat{x}\in \mathbf{R}^d$ is {coherent} if the data model~\eqref{eqn:ML_estimation} with estimate $x=\hat{x}$ yields non-negative observations $Y(\hat{x})$ with probability one, that is
	\begin{align}
	\mathbb{P}(Y(\hat{x})\geq 0)=1.
	\end{align} 
\end{definition}
The problem of computing coherent ML estimators is
\begin{align}
\begin{array}[t]{ll}
\underset{x}{\text{maximize}}
& f_{\Delta}(y-\theta(x;R)) \\
\text{subject to} & \mathbb{P}(Y(x)\geq 0)=1.
\end{array} 
\label{eqn:ML_estimation_relistic}
\end{align}
Our first result is that the constraint $\mathbb{P}(Y(x)\geq 0)=1$ is tractable in the sense that it can be expressed as a set inequality involving target position $x$ and the support set $S_\Delta$.
\begin{lemma}(Rewriting the ML Constraint)
	\label{lemma:rewriting_ML}
	Let $Y(x)$ be a random vector given by model~\eqref{eqn:freq_model} with $x$ a fixed vector. Then
	\begin{align*}
	\mathbb{P}(Y(x)\geq 0)=1 \enspace \Leftrightarrow \enspace \theta(x;R)+S_\Delta\subseteq \mathbf{R}_+^M.
	\end{align*}
\end{lemma}
\begin{proof}
	Fix $x$ and assume that $\theta(x;R)+S_\Delta\subseteq \mathbf{R}_+^M$ holds. Using $\mathbb{P}(\Delta \in S_\Delta)=1$ let us compute the desired probability
	\begin{align*}
	\mathbb{P}(Y(x)\geq 0 )&= \mathbb{P}(Y(x)\geq 0 \cap \Delta \in S_\Delta )\\
	&=\mathbb{P}(\theta(x;R) + \Delta\geq 0 \cap \Delta \in S_\Delta )\\
	&=1.
	\end{align*}
	The last equality uses $\theta(x;R)+S_\Delta\subseteq \mathbf{R}_+^M$. To prove the reverse implication assume that there exists a $\delta^*\in S_\Delta$ such that $\theta_m(x;R)+\delta^*_m<0$ for some $m=1,\dots,M$.  Given $\delta^*$ and $x$ define a radius $\epsilon^*=0.5\big(-\delta^*_m-\theta_m(x;R)\big)>0$. Then
	\begin{align}
	B(\delta^*,\epsilon^*) \cap [-\theta(x;R),+\infty)=\emptyset
	\label{eqn:small_perturb}
	\end{align} 
	with $	B(\delta^*,\epsilon^*)=\{x:  ||x-\delta^*||<\epsilon^*\}$ a (open) Euclidean ball and $[-\theta(x;R),+\infty)=\{v: v_m\geq -\theta_m(x;R) \}$ a multi-dimensional interval. Using~\eqref{eqn:small_perturb} we have
	\begin{align*}
	\mathbb{P}(Y(x)\geq 0 )&= \mathbb{P}(\theta(x;R) + \Delta\geq 0 )\\
	&=\mathbb{P}(\theta(x;R) + \Delta\geq 0 \cup \Delta \in B(\delta^*,\epsilon^*) )\\
	&\enspace \,\, +\mathbb{P}(\theta(x;R) + \Delta\geq 0 \cap \Delta \in B(\delta^*,\epsilon^*))\\
	&\enspace \,\, -\mathbb{P}( \Delta \in B(\delta^*,\epsilon^*))
	\\
	&=\mathbb{P}(\theta(x;R) + \Delta\geq 0 \cup \Delta \in B(\delta^*,\epsilon^*))\\
	&\enspace \,\, -\mathbb{P}( \Delta \in B(\delta^*,\epsilon^*))\\
	&<1.
	\end{align*}
	The third equality is a direct consequence of~\eqref{eqn:small_perturb}. The final inequality uses $\mathbb{P}( \Delta \in B(\delta^*,\epsilon^*))>0$  which follows from an equivalent characterization of $S_\Delta$~\cite{billingsley1995probability} (page 181).
\end{proof}
By Lemma~\ref{lemma:rewriting_ML} we can rewrite problem~\eqref{eqn:ML_estimation_relistic} as
\begin{align}
\begin{array}[t]{ll}
\underset{x}{\text{maximize}}
& f_{\Delta}(y-\theta(x;R)) \\
\text{subject to} & \theta(x;R)+S_\Delta\subseteq \mathbf{R}_+^M
\end{array} 
\label{eqn:ML_estimation_relistic_2}
\end{align}
 Using formulation~\eqref{eqn:ML_estimation_relistic_2} we can already give an interpretation for set $\mathcal{X}$. Assume that the uncertainty region $\Phi$ is bounded and let $f_{\Delta}$ denote the uniform\footnote{The symbol $\mathbf{1}_{\Phi}(\phi)$ in the integral~\eqref{eqn:ML_estimation_relistic_2}  denotes the indicator function of set $\Phi$, that is, $\mathbf{1}_{\Phi}(\phi)=1$ when $\phi\in \Phi$ and $\mathbf{1}_{\Phi}(\phi)=0$ otherwise.} density over $\Phi$
\begin{align}
f_{\Delta}(\delta)=\begin{cases} 0 &\mbox{if } \delta \notin \Phi \\
\frac{1}{\int \mathbf{1}_{\Phi}(\phi)  d\phi} & \mbox{if }\delta \in \Phi. \end{cases}
\label{eqn:uniform_noise}
\end{align}
Then set $\mathcal{X}$ is {precisely} the set of coherent ML estimators so
\begin{align*}
\begin{array}[t]{ll}
\underset{x}{\arg \max}
& f_{\Delta}(y-\theta(x;R)) \\
\text{subject to} & \theta(x;R)+\Phi\subseteq \mathbf{R}_+^M
\end{array} =\mathcal{X}.
\end{align*}
So, for uniform noise, $\mathcal{X}$ is equal to the set of coherent ML estimators. It turns out that set $\mathcal{X}$ also admits an \textit{additional} representation when the underlying noise density is actually unknown. This is the focus of the next section.
\subsection{\textcolor{black}{Estimation with Minimal Noise Knowledge}}
\textcolor{black}{Assume now we are interested in solving problem~\eqref{eqn:ML_estimation_relistic_2}, but we have only partial knowledge on the noise vector $\Delta$. In concrete, we assume that only the support of $\Delta$ is known, that is, we know that the noise is supported on a uncertainty region $\Phi$, so $S_\Delta=\Phi$. This assumption is, essentially, the minimal knowledge we can assume for the noise vector $\Delta$ since the actual noise density $f_{\Delta}$ is completely arbitrary. For example, if the support $\Phi$ is an ellipsoid $\mathcal{E}(0,\Sigma)$, then we are only assuming that, with probability one, any noise realization comes from $\mathcal{E}(0,\Sigma)$. We are not assuming that the noise realizations are more likely to come from any particular region of  $\mathcal{E}(0,\Sigma)$ since the actual noise density $f_{\Delta}$ is arbitrary. Assuming that the noise density $f_{\Delta}$ is unknown is a main focus of the robust statistics field~\cite{huber2004robust}.}
\\~\\
 Given $\Phi$, assume that the density $f_{\Delta}$ is unknown but we know that it belongs to a family of densities $\mathcal{F}$. By imposing some natural restrictions on family $\mathcal{F}$, we show that set  $\mathcal{X}$ is also equal to the set of coherent ML estimators, given that the noise density belongs to family $\mathcal{F}$. Let $\mathcal{D}_M$ denote the set of densities in $\mathbf{R}^M$, 
\begin{align*}
\mathcal{D}_M:=\Big\{ & f:\int_{\mathbf{R}^M} f(t) dt=1, f(t)\geq 0 \enspace \forall t\in \mathbf{R}^M \Big\}.
\end{align*}
\begin{theorem}(Frequentist Robustness of $\mathcal{X}$)
	\label{theorem:freq_robust_X}
	Let $\Phi$ denote a closed, convex uncertainty region with a non-empty interior. Let $\mathcal{F}$ denote the set of densities $f_{\Delta}$ that are positive only inside the support $\Phi$ and such that the ML objective $f_{\Delta}(y-\theta(x;R))$ is not identically equal to zero for $\theta(x;R)+\Phi\subseteq \mathbf{R}_+^M $,
	\begin{align*}
	\mathcal{F}:=\Big\{ &f_{\Delta}:\delta\notin \Phi \Rightarrow f_{\Delta}(\delta)=0,\enspace f_{\Delta}\in \mathcal{D}_M,\enspace S_\Delta=\Phi\\
	& \exists \, x^*: f_{\Delta}(y-\theta(x^*;R))>0,\, \theta(x^*;R)+\Phi\subseteq \mathbf{R}_+^M   \Big\}.
	\end{align*}
	Then  $\mathcal{X}$ is a tight majorizer of the set of \textit{coherent} ML estimators parametrized by the noise density $f_{\Delta}\in \mathcal{F}$, i.e.,
	\begin{align}
	\bigcup_{f_\Delta \in \mathcal{F}}
	\begin{array}[t]{ll}
	\underset{x}{\arg \max}
	& f_{\Delta}(y-\theta(x;R)) \\
	\text{subject to} & \theta(x;R)+\Phi\subseteq \mathbf{R}_+^M 
	\end{array} = \mathcal{X}.
	\label{eqn:majorization}
	\end{align}
\end{theorem}
\begin{proof} 
	(Sufficiency)
	Since $f_{\Delta}\in \mathcal{F}$ there exists a $x^*$ such that $f_{\Delta}(y-\theta(x^*;R))>0$ and $\theta(x^*;R)+\Phi\subseteq \mathbf{R}_+^M $. So $x^*$ is feasible for~\eqref{eqn:ML_estimation_relistic_2}. Assume that $x$ is a maximizer of~\eqref{eqn:ML_estimation_relistic_2} but $x\notin \mathcal{X}$. Since  $\theta(x;R)+\Phi\subseteq \mathbf{R}_+^M $ is must be that $y-\theta(x;R)\notin \Phi$. Using the first property of family $\mathcal{F}$  we get $f_{\Delta}(y-\theta(x;R))=0$. But this is impossible since $x$ maximizes~\eqref{eqn:ML_estimation_relistic_2} but $x^*$ is a feasible point which renders a larger objective, i.e., $\theta(x^*;R)+\Phi\subseteq \mathbf{R}_+^M $ and  $f_{\Delta}(y-\theta(x^*;R))>0=f_{\Delta}(y-\theta(x;R))$. So any maximizer $x$ must belong to $\mathcal{X}$ regardless of $f_{\Delta}\in \mathcal{F}$.
	~\\
	(Necessity)
	We must show that for every $\hat{x}\in \mathcal{X}$ there exists a $f_{\Delta}\in \mathcal{F}$ such that $\hat{x}$ is optimal for problem~\eqref{eqn:ML_estimation_relistic_2}. Take $f_{\Delta}$ has the standard truncated Gaussian density with mean $y-\theta(\hat{x};R)$,
	\begin{align}
	f_{\Delta}(\delta)=  \frac{\exp\left(-\frac{\|\delta-\{y-\theta(\hat{x};R)\}\|^2}{2}\right)}{\int_\Phi \exp\left(-\frac{\|\hat{\delta}-\{y-\theta(\hat{x};R)\}\|^2}{2}\right) d\hat{\delta} } \mathbf{1}_{\Phi}(\delta). \label{eq:truncated_gaussian}
	\end{align}
	The factor $\int_\Phi \exp\left(-\frac{\|\hat{\delta}-\{y-\theta(\hat{x};R)\}\|^2}{2}\right) d\hat{\delta}$ is strictly positive and finite since $\Phi$ has a non-empty interior and the gaussian density $\delta \mapsto \exp\left(-\frac{\|\delta\|^2}{2} \right) $ is strictly positive and continuous. Using~\eqref{eq:truncated_gaussian} and $\hat{x}\in \mathcal{X}$ we get that $\hat{x}$ solves~\eqref{eqn:ML_estimation_relistic_2}. Now the claim $S_{\Delta}=\Phi$ follows because  $\Phi$ is closed, convex and $\text{int } \Phi\neq \emptyset$ (we omit the proof due to space constraints) and the remaining conditions on density $f_{\Delta} $ are direct to verify. Note that 
	the former argument works if we truncate any continuous density $f$ that is strictly positive in $\mathbf{R}^M$ $(f>0)$ and attains its maximum at the desired point $y-\theta(\hat{x};R)$.
\end{proof}
\begin{remark}
	Simple examples show that the two new restrictions imposed on family $\mathcal{F}$ are actually necessary for equality~\eqref{eqn:majorization}. If one of those conditions fails then there exist noise densities $f_{\Delta}$  and support sets $\Phi$ such that the set of \textit{coherent} ML estimators is not contained in $\mathcal{X}$, that is, the left-hand side of~\eqref{eqn:majorization} becomes strictly smaller than the right-hand side. Note also that these conditions naturally generalize the uniform case: when $f_{\Delta}$ is given by~\eqref{eqn:uniform_noise} the first constraint on $\mathcal{F}$ imposes that the density $f_{\Delta}$ is not changed point-wise\footnote{By definition, densities are unique up to a almost everywhere equivalence. For example let $\delta^*\in \mathbf{R}^M$ denote a fixed vector. If $f_{\Delta}$ is a density for $\Delta$ then the function $\hat{f}_{\Delta}$ defined by $\hat{f}_{\Delta}:=f_{\Delta} \mathbf{1}_{\Phi\setminus\{\delta^*\}}$ is also a valid density for $\Delta$ since $\hat{f}_{\Delta}$ is non-negative and functions $f_{\Delta}$ and $\hat{f}_{\Delta}$ differ only in the singleton $\{\delta^*\}$ which has zero Lebesgue measure. In simple terms, the first condition of theorem~\ref{theorem:freq_robust_X} avoids these degenerate modifications of noise densities.} such that $f_{\Delta}(\delta)>0$ for $\delta\notin \Phi$; the condition $\exists \, x^*: f_{\Delta}(y-\theta(x^*;R))>0,\, \theta(x^*;R)+\Phi\subseteq \mathbf{R}_+^M  $ is simply imposing  that $\mathcal{X}$  is non-empty.
	~\\
	\textcolor{black}{In conclusion, the restrictions imposed on family $\mathcal{F}$ are unavoidable for equality~\eqref{eqn:majorization} and, essentially, they impose that the ML problem is not ill-posed and that the noise densities are not degenerate (see footnote six).}
\end{remark}

	\textcolor{black}{We now summarize the main finding of this section for the overall problem of target localization:
\begin{itemize}
	\item In estimation problems, we often have minimal knowledge on the noise vector $\Delta$. In concrete, the noise distribution $f_{\Delta}$ is often unknown but we can known the support  $S_{\Delta}$ of $\Delta$. In this case, any ML estimate $\hat{x}$ of $x$ lies in set $\mathcal{X}$, regardless of the underlying noise density $f_{\Delta}\in \mathcal{F}$. Furthermore, any point $\hat{x}\in \mathcal{X}$ is a plausible estimate of $x$ in the following sense: there exists a noise distribution $f_{\Delta}\in \mathcal{F}$ such that, if $f_{\Delta}$ is the true distribution of $\Delta$ then $\hat{x}$ is a Maximum-Likelihood estimate of $x$. So, set $\mathcal{X}$ is of primary interest in robust statistics since it is tracking a set of statistical estimators of target $x$ \textit{regardless} of the underlying noise density.
\end{itemize}
The next section provides some background on linear-fractional representations: a fundamental tool in deriving our relaxation.
}

\section{Background on Linear-Fractional Representations}
\label{sec:backgroundLFR}

For our purposes, a linear-fractional representation (LFR) is a map that transforms matrices to vectors, denoted by
$$U \in {\mathcal U}\,\, \mapsto \,\, \begin{bmatrix}
     \begin{array}{@{}c|c@{}}
  C & d \\
  \hline
  B & a
    \end{array}
  \end{bmatrix}_{[p]}(U),$$
 where  $p$ is a positive integer and
\begin{equation} 
\begin{bmatrix}
     \begin{array}{@{}c|c@{}}
  C & d \\
  \hline
  B & a
    \end{array}
  \end{bmatrix}_{[p]}(U) =  B (I_p \otimes U )\left( I_q - C ( I_p \otimes U) \right)^{-1} d + a. \label{eq:LFR} \end{equation}
Here, $I_n$ is the $n\times n$ identity matrix and $\otimes$ denotes the Kronecker-product of matrices, thus
 $$I_p \otimes  U = \begin{bmatrix} U \\ & \ddots \\ & & U \end{bmatrix} \text{ ($p$ copies)}.$$ In~\eqref{eq:LFR}, $q$ is the number of rows of $C$. The LFR is determined by the positive integer $p$ and by the matrices $B$ and $C$, and the vectors $a$ and $d$, and is assumed to be well-posed on its domain ${\mathcal U}$ (meaning that matrix $I - C (I_p \otimes U)$ is assumed to be invertible for each matrix $U$ in the set ${\mathcal U}$).

The image of the LFR is the set of vectors of the form  \begin{equation} \text{Im } \begin{bmatrix}
     \begin{array}{@{}c|c@{}}
  C & d \\
  \hline
  B & a
    \end{array}
  \end{bmatrix}_{[p]} = \left\{ \begin{bmatrix}
      \begin{array}{@{}c|c@{}}
  C & d \\
  \hline
  B & a
    \end{array}
  \end{bmatrix}_{[p]}(U) \colon U \in {\mathcal U}\right\}. \label{eq:imLFR} \end{equation}

\paragraph*{A simple example} Many maps can be phrased
as LFRs. To illustrate, consider the non-linear rational mapping $$u \in {\mathcal U} = [-1,1] \quad \mapsto \quad \begin{bmatrix} 1 - u^2 \\ 4/(2 + u ) \end{bmatrix}.$$
This map can be phrased as the LFR 
\begin{equation} u \in {\mathcal U} \quad \mapsto \quad  \begin{bmatrix}
     \begin{array}{@{}c|c@{}}
  C & d \\
  \hline
  B & a
    \end{array}
  \end{bmatrix}_{[p]}(u), \label{eq:simpleLFR} \end{equation}
with $p = 3$ and
\begin{align*}
 C = \begin{bmatrix} 0 & 1 & 0 \\ 0 & 0 & 0 \\ 0 & 0 & -\frac{1}{2}\end{bmatrix} \hspace{-0.1cm}, \, B = \begin{bmatrix} -1 & 0 & 0 \\ 0 & 0 & -1\end{bmatrix} \hspace{-0.1cm}, \, d = \begin{bmatrix} 0 \\ 1 \\ 1\end{bmatrix} \hspace{-0.1cm},\, a = \begin{bmatrix} 1 \\ 2\end{bmatrix}.
\end{align*}
\paragraph*{Re-parameterization of the image of an LFR} A main technique when dealing with LFRs is to re-parameterize their image. This technique takes the image of an LFR, which is parameterized in~\eqref{eq:imLFR} in terms of the variable  $U$, and re-parameterizes the image in terms of a new variable $v$. With this re-parameterization, the image of an LFR becomes easier to handle. To obtain the re-parameterization, first note that
\vspace{0.2cm}
~\\
{\small$\begin{aligned}
\text{Im} \begin{bmatrix}
     \begin{array}{@{}c|c@{}}
  C & d \\
  \hline
  B & a
    \end{array}
  \end{bmatrix}_{[p]}
  & =  \left\{  B (I_p \otimes U ) (I - C (I_p \otimes U) )^{-1} d  + a \colon U \in {\mathcal U} \right\} \nonumber  \\
   & = \begin{aligned}[t] \left\{ B v + a  \colon v \right. &=  (I_p \otimes U) w,  \\  w &=  (I - C (I_p \otimes U))^{-1} d, U \in {\mathcal U} \left. \right\}  \end{aligned} \nonumber \\ & =  \left\{ B v  + a \colon v = (I_p \otimes U) w,\, w = C v + d, \, U \in {\mathcal U} \right\}\label{eq:exLFR}
\end{aligned}$}
\vspace{0.05cm}
~\\
which expresses the image of the LFR in terms of an extended space with three variables: matrix $U$, and new vectors $v$, $w$.

The next step flattens this space  by removing $U$; we refer to this step as the \textit{flattening step}. This step assumes  that the pairs of vectors $(v, w)$ satisfying $v = (I_p \otimes U) w$ for some $U \in {\mathcal U}$ can be written as the inverse image of a positive semidefinite cone under a linear map  that acts on the outer product $$\begin{bmatrix} v \\ w \end{bmatrix} \begin{bmatrix} v \\ w \end{bmatrix}^T.$$
That is, the flattening step assumes that the set 
\begin{equation} \left\{ (v, w) \colon v = (I_p \otimes U) w, U \in {\mathcal U} \right\} \label{eq:fm1} \end{equation} can be written as \begin{equation}\left\{ (v, w) \colon L_{{\mathcal U}_{[p]}}\left( \begin{bmatrix} v \\ w \end{bmatrix} \begin{bmatrix} v \\ w \end{bmatrix}^T\right) \succeq 0 \right\}, \label{eq:fm2} \end{equation}
where $L_{{\mathcal U}_{[p]}}$ is a linear map from the set of symmetric matrices ${\mathbf S}^{{n_v + n_w}}$ to the set of symmetric matrices ${\mathbf S}^n$, with $n_v$ and $n_w$ being the size of $v$ and $w$, respectively, and $n$ depending on the particular LFR at hand;  the notation $X \succeq 0$ means that the symmetric matrix $X$ is positive semidefinite. The map $L_{{\mathcal U}_{[p]}}$, which we refer to as the flattening map,  depends on the positive integer $p$ and on the domain ${\mathcal U}$, and has to be worked out  from LFR to LFR. 

The flattening map $L_{{\mathcal U}_{[p]}}$ allows to rewrite the image set as 
{\small\begin{align}
\text{Im} \begin{bmatrix}
  \begin{array}{@{}c|c@{}}
  C & d \\
  \hline
  B & a
    \end{array}
  \end{bmatrix}_{[p]} & = 
\Big\{ B v + a \colon  w = C v + d,   L_{{\mathcal U}_{[p]}} \hspace{-0.05cm} \bigg( \begin{bmatrix} v \\ w \end{bmatrix} \begin{bmatrix} v \\ w \end{bmatrix}^T \hspace{-0.05cm}\bigg) \succeq 0 \hspace{-0.02cm} \Big\} \nonumber \\ & =  \left\{ B v + a  \colon  L_{{\mathcal U}_{[p]}}\left( P \begin{bmatrix} v \\ 1 \end{bmatrix} \begin{bmatrix} v \\ 1 \end{bmatrix}^T P^T \right) \succeq 0 \right\},\label{eq:exLFR2}
\end{align}}
where the last equality eliminates variable $w$ and defines  $$P = \begin{bmatrix} I & 0 \\ C & d \end{bmatrix}.$$

In~\eqref{eq:exLFR2}, the image of the LFR is now parametrized by $v$. 



\paragraph*{A simple example (cont.)} To illustrate how the re-parameterization plays out, let us return to the simple LFR in~\eqref{eq:simpleLFR}. To obtain the flattening map for this LFR, note that
\begin{eqnarray}
  \lefteqn{v = (I_3 \otimes u ) w, \text{ for some }u \in {\mathcal U} = [-1,1]} \nonumber \\ & \Leftrightarrow & v = u w, \text{ for some } -1 \leq u \leq 1 \nonumber \\ & \Leftrightarrow & v v^T \preceq w w^T. 
  \label{eqn:aux_example}
\end{eqnarray}
The flattening map is therefore $L_{{\mathcal U}_{[3]}} \colon {\mathbf S}^6 \rightarrow {\mathbf S}^3$, $${\mathcal L}_{{\mathcal U}_{[3]}}\left(\begin{bmatrix} S_{11} & S_{12} \\ S_{21} & S_{22} \end{bmatrix} \right)= S_{22} - S_{11}.$$


\section{Our approach}
\label{sec:our_approach}
Equipped with the toolset of LFRs, we now return to the core problem~\eqref{eq:core1}.
We do a sequence of  reformulations  to arrive at our convex relaxation. Our reformulations are such that the optimal value of problem~\eqref{eq:core1} remains the same up until step \textit{g}) when we drop an underlying (non-convex) rank constraint. Our outer approximation of the set $\mathcal{X}$ is denoted by $\overline {\mathcal X}_{\text{LFR}}$.

\paragraph{Reformulate with quadratics}
We start by rewriting~\eqref{eq:core1} with quadratic constraints,
\begin{equation}
\begin{array}[t]{ll} \underset{x, z, u}{\text{maximize}} & s^T x \\ \text{subject to} & \left\| x \right\|^2 - z = 0\\
& z - 2 r_m^T x + \left\| r_m \right\|^2 - \Sigma_{mm} \geq 0 \\
& y_m  - u_m \geq 0,\enspace u^T \Sigma^{-1} u \leq 1 \\ &  y_m^2  - \left\| r_m \right\|^2 - z  + 2 r_m^T x - 2 y_m u_m + u_m^2 =  0.
\end{array} 
\label{eq:core2}
\end{equation}
To obtain~\eqref{eq:core2}, we first swapped the data constraint $y_m - u_m = \left\| x - r_m \right\|$ for the equivalent  the pair of conditions $y_m - u_m \geq 0$ and $( y_m - u_m) ^2 = \left\| x - r_m \right\|^2$; then, we expanded the squares and introduced the new variable $z = \left\| x \right\|^2$.

\paragraph{Lift uncertain vector $u$ to uncertain matrix $U$} Next, we view the vector $u$ as the first column of a square matrix $U$, that is, $u = U e_1$, where $e_1$ is the first column of $I_M$,  \begin{equation} I_M = \begin{bmatrix} e_1 & e_2 & \cdots & e_M \end{bmatrix}. \label{eq:IM} \end{equation}
Accordingly, we write~\eqref{eq:core2} in terms of the variable $U$:
\begin{equation}
\begin{array}[t]{ll} \underset{x, z, U}{\text{maximize}} & s^T x \\ \text{subject to} & \left\| x \right\|^2 - z = 0\\
& z - 2 r_m^T x + \left\| r_m \right\|^2  - \Sigma_{mm} \geq 0 \\
& y_m  - (U e_1)_m \geq 0\\
&  y_m^2  - \left\| r_m \right\|^2   + 2 r_m^T x \\ & - z - 2 y_m (U e_1)_m  + (U e_1)_m^2 =  0 
 \\ & \left\| \Sigma^{-1/2} U \right\| \leq 1,
\end{array} 
\label{eq:core3}
\end{equation}
where $\left\| X \right\|$ is the spectral norm (maximum singular value) of matrix $X$. \textcolor{black}{ Problems~\eqref{eq:core2} and~\eqref{eq:core3} share the same optimal value for two reasons. First if the triplet $(x,z,U)$ is feasible in~\eqref{eq:core3}, then the assignment $(x,z,Ue_1)$ yields a feasible point in~\eqref{eq:core2} since the first column  of $U$ belongs to  $\mathcal{E}(0,\Sigma)$, that is
\begin{align*}
1\geq \left\| \Sigma^{-1/2} U\right\|^2= \max_{ ||w||=1  } w^T U^T \Sigma^{-1} U w \geq e_1^T U^T U e_1.
\end{align*}
Second, if the triplet $(x,z,u)$ is feasible in~\eqref{eq:core2} then consider the assignment $(x,z,\overline{U})$  with $\overline{U}$ retaining the first column of $U$, that is,  $\overline{U}=(Ue_1\enspace 0)$. Point  $(x,z,\overline{U})$  a feasible point in~\eqref{eq:core3} because the spectral norm $||.||$ is equal to maximum singular value, that is
\begin{align*}
||\Sigma^{-1/2} \overline{U}||&=\lambda_{\text{max}} \Big\{\overline{U}^T\Sigma^{-1} \overline{U}\Big\}\\
&=\lambda_{\text{max}} \Big\{ \begin{bmatrix}
e_1^T {U}  \\
0
\end{bmatrix} \Sigma^{-1} \begin{bmatrix}
{U} e_1 & 0
\end{bmatrix}\Big\} \\
&=\lambda_{\text{max}} \Big\{ \begin{bmatrix}
e_1^T {U} \Sigma^{-1} U e_1 & 0  \\
0 & 0
\end{bmatrix} \Big\}\\
&=e_1^T {U} \Sigma^{-1} U e_1\\
&\leq 1,
\end{align*}
with  $\lambda_{\text{max}}(A)$ denoting the maximum singular value of a symmetric matrix $A$.}
We switch from $u$ to $U$ because it makes the flattening map of the forthcoming LFR easier to compute.

\paragraph{Write last three constraints in terms of the image of a nonlinear map} By introducing the map $\phi_{x,z} \colon {\mathcal U} \rightarrow {\mathbf R}^{2M}$, \begin{equation} U \,\, \mapsto \,\, \begin{bmatrix} y_m  - (U e_1)_m \\  y_m^2  - \left\| r_m \right\|^2  + 2 r_m^T x  - z - 2 y_m (U e_1)_m + (U e_1)_m^2  \end{bmatrix}, \label{eq:defphi} \end{equation}
where \begin{equation}{\mathcal U} = \left\{ U \in {\mathbf R}^{M\times M} \colon \left\| \Sigma^{-1/2} U \right\| \leq 1 \right\}, \label{eq:setU} \end{equation} we can remove the variable $U$ from~\eqref{eq:core3} and interpret the last three constraints of~\eqref{eq:core3} as  restricting the image of the map~$\phi_{x,z}$, in fact, as saying that the image of $\phi_{x,z}$ must intersect ${\mathbf R}^M_+ \times \{ 0_M \}$, where $0_M = ( 0, \ldots, 0)$ is the $M$-dimensional vector with all entries equal to $0$:
\begin{equation}
\begin{array}[t]{ll} \underset{x, z}{\text{maximize}} & s^T x \\ \text{subject to} & \left\| x \right\|^2 - z = 0\\ 
& z - 2 r_m^T x + \left\| r_m \right\|^2 - \Sigma_{mm} \geq 0 \\ & \text{Im}\,\phi_{x,z} \cap {\mathbf R}^M_+ \times \{ 0_M \} \neq \emptyset.\end{array} 
\label{eq:core4}
\end{equation}
The remaining steps use LFR techniques to express the uncertainty matrix $U$ as a flattening inequality in terms of the target $x$ and its squared norm $z$. In short, our approach models measurement uncertainty \textit{implicitly} as a relation on $x$ and $z$.

\paragraph{Phrase the map $\phi_{x,z}$ as an LFR} We now express the map $\phi_{x,z}$ as an LFR. Specifically, we have
\begin{equation}\phi_{x,z}(U) = \begin{bmatrix}
  \begin{array}{@{}c|c@{}}
  C & d \\
  \hline
B_1 & a_1 \\ B_2 & a_2  
    \end{array}
  \end{bmatrix}_{[2M]}(U), \label{eq:phiLFR} \end{equation}
where
\begin{align}
 C &=  E \left( I_M \otimes \begin{bmatrix} 0 & 1 \\ 0 & 0  \end{bmatrix} \right) F & d &= E \left( 1_M \otimes \begin{bmatrix} 0 \\ 1 \end{bmatrix} \right) \label{eq:matLFR}\\ 
 B_1 & =  \left( I_M \otimes \begin{bmatrix} 0 & -1 \end{bmatrix} \right) F & a_1 & = y \label{eq:LFRrow1} \\
 B_2 & = \widehat{B}_2 F & a_2 & = q + 2 R^T x - z 1_M. \label{eq:LFRrow2}
\end{align}
Here, $$\widehat{B}_2= \begin{bmatrix} 1 -2y_1 \\ & \ddots \\ & & 1 -2y_M \end{bmatrix},$$
$R = \begin{bmatrix} r_1 & \cdots & r_M \end{bmatrix}$ is the matrix that displays the positions of the $M$ reference landmarks in its columns, and
 \begin{align}q = \begin{bmatrix} y_1^2 - \left\| r_1 \right\|^2 \\ \vdots \\ y_M^2 - \left\| r_M \right\|^2 \end{bmatrix}
 \label{eq:definition_q}
 \end{align}
 is an auxiliary vector.
Finally, we define $E = I_{2M} \otimes e_1$ and $$F = \begin{bmatrix} I_2 \otimes e_1^T \\ & I_2 \otimes e_2^T \\ & &  \ddots  \\ & & & I_2 \otimes e_M^T \end{bmatrix},$$
with $e_m$ the $m$th column of the identity matrix $I_M$ (see~\eqref{eq:IM}). 

To express the map $\phi_{x,z}$~\eqref{eq:defphi} as the LFR~\eqref{eq:phiLFR} we use simple properties of LFRs, as detailed in appendix~\ref{ap:step1}. The LFR in~\eqref{eq:phiLFR} is well defined for any matrix $U\in \mathbf{R}^{M\times M}$ since the matrix  $I - C (I_{2M} \otimes U)$ is invertible for any $U$ (see appendix~\ref{ap:step1}).

\paragraph{Re-parameterize the image of the LFR}
Our next step is to re-parameterize the image of the  LFR~\eqref{eq:phiLFR} by applying the technique of Section~\ref{sec:backgroundLFR}. The only non-obvious point is the computation of the flattening map, defined in~\eqref{eq:fm1} and~\eqref{eq:fm2}. 

For the LFR at hand, it turns out that 
\begin{equation} \left\{ (v, w) \colon v = ( I_{2M} \otimes  U ) w, U \in {\mathcal U} \right\}, \label{eq:fmap1} \end{equation} 
can be written as \begin{equation}\left\{ (v, w) \colon L_{{\mathcal U}_{[2M]}}\left( \begin{bmatrix} v \\ w \end{bmatrix} \begin{bmatrix} v \\ w \end{bmatrix}^T\right) \succeq 0 \right\}, \label{eq:fmap2} \end{equation}
where the flattening map $L_{{\mathcal U}_{[2M]}} \colon {\mathbf S}^{4 M^2} \rightarrow {\mathbf S}^{2M}$ is 
\begin{equation}
L_{{\mathcal U}_{[2M]}}\left( \begin{bmatrix} S_{11} & S_{12} \\ S_{21} & S_{22} \end{bmatrix} \right) = \sum_{m = 1}^M E_m S_{22} E_m^T - F_m S_{11} F_m^T,
  \label{eq:thefmap}
  \end{equation}
with $E_m = I_{2M} \otimes e_m^T$ and $F_m = I_{2M} \otimes e_m^T \Sigma^{-1/2}$. Each matrix $S_{ij}$ has size $2M^2 \times 2M^2$. 
The details of this step are given in Appendix~\ref{ap:step2}. Plugging the image of the LFR, re-parameterized as in~\eqref{eq:exLFR2}, into problem~\eqref{eq:core4} allows to derive the reformulation
\begin{equation}
\begin{array}[t]{ll} \underset{x, z, v}{\text{maximize}} & s^T x \\ \text{subject to} & \left\| x \right\|^2 - z = 0\\
& z - 2 r_m^T x + \left\| r_m \right\|^2 - \Sigma_{mm} \geq 0 \\ & 
B_1 v + y \geq 0,\enspace B_2 v + q + 2 R^T x - z 1_M  = 0 \\ & L_{{\mathcal U}_{[2M]}}\left(
 \begin{bmatrix} v \\ C v + d \end{bmatrix} \begin{bmatrix} v \\ C v + d \end{bmatrix}^T   \right) \succeq 0,
\end{array} 
\label{eq:core5}
\end{equation}
where $L_{{\mathcal U}_{[2M]}}$ is as~\eqref{eq:thefmap}, $C$ and $d$ are given by~\eqref{eq:matLFR}, $B_1$, $B_2$ come from~\eqref{eq:LFRrow1} --~\eqref{eq:LFRrow2}, $R = \begin{bmatrix} r_1 & \cdots & r_M \end{bmatrix}$ and $q$ equals~\eqref{eq:definition_q}.

\paragraph{Rewrite the last problem in terms of a matrix with rank $1$} We now rewrite~\eqref{eq:core5} in terms of the rank $1$ matrix
\begin{equation} X = \begin{bmatrix} x  \\ z  \\ v  \\ 1 \end{bmatrix}  \begin{bmatrix} x  \\ z  \\ v  \\ 1 \end{bmatrix}^T. \label{eq:tform} \end{equation}

The objective can be written as $s^T x = \text{Tr}\left( S X \right)$, where $$S = \begin{bmatrix} 0 & 0 & 0 & s/2 \\    & 0 & 0 & 0 \\ & &  0 & 0 & \\ & & & 0 \end{bmatrix}.\label{eqn:objective}$$
Here, $\text{Tr}$ denotes the trace of a matrix, and, to simplify notation, we display from now on only the upper-part of symmetric matrices (and also omit the size of the zero blocks). The first constraint $\left\| x \right\|^2 - z  =  0$ can be written as $\text{Tr}\left( K X \right) = 0$, where $$K = \begin{bmatrix} I_M & 0 & 0 & 0 \\ & 0 & 0 & -1/2 \\ & & 0 & 0 \\  & & & 0 \end{bmatrix}.$$
The second constraint $z - 2 r_m^T x + \left\| r_m \right\|^2 - \Sigma_{mm} \geq  0$ is a linear inequality that is equivalent to $\text{Tr}\left( L_m X \right) \geq 0$, where
$$L_m = \begin{bmatrix} 0 & 0 & 0 & -r_m \\ & 0 & 0 & 1/2 \\ &  & 0 & 0 & \\ & & & \left\| r_m \right\|^2 - \Sigma_{mm} \end{bmatrix}.$$
The third constraint is $B_1 v + y \geq 0$. This linear vectorial constraint is equivalent to $M$ scalar inequalities of the form $$g_m^T v + y_m \geq 0,$$ where $B_{1,m}^T$ is the $m$th row of $B_1$. Each such constraint is written as $\text{Tr}\left(G_m X \right) \geq 0$, where $G_m$ is the symmetric matrix
$$G_m = \begin{bmatrix} 0 & 0 & 0 & 0 \\ & 0 & 0 & 0 \\ & & 0 & B_{1,m} / 2 \\ & & & y_m \end{bmatrix}.$$ The fourth constraint is $$\underbrace{\begin{bmatrix} 2 R^T & -1_M &  B_2 & q \end{bmatrix}}_{H} \begin{bmatrix} x \\ z \\ v \\ 1 \end{bmatrix} = 0,$$
which corresponds to $M$ constraints $\text{Tr}\left( H_m X \right) = 0$, where $H_m = h_m h_m^T$ and $h_m^T$ is the $m$th row of matrix $H$. Finally, the last constraint is a linear matrix inequality (LMI) given by $$L_{{\mathcal U}_{[2M]}}\left( \begin{bmatrix} 0 & P \end{bmatrix} X \begin{bmatrix} 0 \\ P^T  \end{bmatrix} \right) \succeq 0.$$ In sum, problem~\eqref{eq:core5} corresponds to 
\begin{equation}
\begin{array}[t]{ll} \underset{X}{\text{maximize}} & \text{Tr}( S X ) \\ \text{subject to} & \text{Tr}( K X ) = 0,\enspace \text{Tr}( L_m X ) \geq 0 \\
& \text{Tr}( G_m X ) \geq 0,\enspace \text{Tr}( H_m X ) = 0 \\\
& L_{{\mathcal U}_{[2M]}}\left( \begin{bmatrix} 0 & P \end{bmatrix} X \begin{bmatrix} 0 \\ P^T  \end{bmatrix} \right) \succeq 0 \\
& X \succeq 0,\enspace f^T X f = 1,\enspace \text{rank}(X) = 1.
\end{array} 
\label{eq:core6}
\end{equation}
Here, $f$ is the vector with all components equal to $0$, except the last one, which is equal to $1$: $f = ( 0, \ldots, 0 , 1)$. Thus, $f^T X f$ gives the entry in the bottom right corner of $X$. The last three constraints in~\eqref{eq:core6} encode the set of rank one matrices with the bottom right entry equal to one, that is, matrices $X$ as in~\eqref{eq:tform}.

\paragraph{Drop the rank constraint} Removing the rank constraint in~\eqref{eq:core6} leaves a convex semidefinite program (SDP). That SDP is our convex relaxation for the core problem~\eqref{eq:core1}.

\section{Benchmark Approach}
\label{sec:benchmark}
 The benchmark approach~\cite{ZQ_luo} goes through an alternative sequence of steps that also preserve the optimal value of problem~\eqref{eq:core1} up until a  (different) rank constraint is also relaxed. The outer-approximation computed via the benchmark SDP relaxation is denoted by $\overline {\mathcal X}_{\text{SDP}}$.

\paragraph{Reformulate with quadratics} The first step is the same as our proposed approach and creates the reformulation~\eqref{eq:core2}.

\paragraph{Rewrite the last problem in terms of a matrix with rank~$1$}
the next step is to rewrite~\eqref{eq:core2} in terms of the matrix 
\begin{equation} X = \begin{bmatrix} x \\ z \\ u \\ 1 \end{bmatrix} \begin{bmatrix} x \\ z \\ u \\ 1 \end{bmatrix}^T. \label{eq:sa1} \end{equation}
Note that, here, matrix $X$ has dimension $d+2+M$ while, in our LFR relaxation, the underlying $X$ matrix has dimension $d+2+2M^2$. For simplicity, we use the letter $X$ to refer to both matrices. The underlying context will dictate which object is being mentioned. 
Problem~\eqref{eq:core2} is equivalent to
\begin{equation}
\begin{array}[t]{ll} \underset{X}{\text{maximize}} & \text{Tr}( \hat{S} X ) \\ \text{subject to} & \text{Tr}( \hat{K} X ) = 0,\enspace \text{Tr}( \hat{L}_m X ) \geq 0 \\
& \text{Tr}( \hat{G}_m X ) \geq 0 ,\enspace \text{Tr}( \hat{H}_m X ) = 0,\enspace  \text{Tr}( \hat{J} X ) \geq 0 \\
& X \succeq 0 ,\enspace \hat{f}^T X \hat{f} = 1,\enspace \text{rank}(X) = 1.
\end{array} 
\label{eq:coress}
\end{equation}
where matrices $ \hat{S},\hat{K},\hat{L}_m,\hat{G}_m,\hat{H}_m,\hat{J},\hat{f}$ are obtained by proceeding as in Section~\ref{sec:our_approach} step \textit{f}), that is, rewriting the objective and constraints of~\eqref{eq:core2} in terms of matrix $X$ composed with linear mappings. We omit these derivations due to space constraints, but the  matrices are given in Appendix~\ref{ap:aux_matrices}.

\paragraph{Drop the rank constraint} Dropping the rank constraint in~\eqref{eq:coress} produces a standard relaxation for problem~\eqref{eq:core1}.

\section{Numerical Results}
\label{sec:numerical_results}
\label{sec:random_target}
We compare the rectangle obtained\footnote{All experiments were developed using the package CVX with MATLAB$^\circledR$ \cite{cvx}, \cite{gb08}. In particular we have
	used version 4.0 of the solver SDPT3
	in a computer with a Intel(R) Core(TM) i7-3630QM CPU @ 2.4GHz processor.} by our approach,
$\overline {\mathcal X}_{\text{LFR}}$, with the rectangle obtained by the standard SDP relaxation, $\overline {\mathcal X}_{\text{SDP}}$, in a two dimensional localization scenario, that is $d=2$, with $M=3$ reference landmarks. Thus,
$\overline {\mathcal X}_{\text{LFR}}$ is obtained by solving four problems of the
form~\eqref{eq:core6} (without the rank constraint) for  $s\in \{ (0,1),(1,0),(0,-1),(-1,0)  \}$. Set  $\overline {\mathcal X}_{\text{SDP}}$ is obtained via~\eqref{eq:coress} (without the rank constraint).
\vspace*{0.2ex}
\\~\\
\noindent\textbf{Simulated setup.} In order to quantify the amount of noise in measurement vector $y$ assume that, given the true target position $x^*$, any noise vector $u\in \mathcal{E}(0,\Sigma)$ changes the true distances $||x^*-r_i||$ by at most $\alpha\%$, that is, the noisy measurements $y_i=||x^*-r_i||+u_i$ are contained in a symmetric interval centered at the true measurement $||x^*-r_i||$ and of length $2 \alpha||x^*-r_i||$, regardless of $u\in \mathcal{E}(0,\Sigma)$. In compact notation:
	\begin{equation}
	\mathcal{E}(\theta(x^*;R),\Sigma) \subseteq \Big[ (1-\alpha) \theta(x^*;R) , (1+\alpha) \theta(x^*;R)  \Big].
	\label{eqn:noise_level_ellip}
	\end{equation}
with $\theta(x^*;R)$ denoting the vector of true distance measures defined in~\eqref{eq:true_measurements}. In this setup, scalar $\alpha\in (0,1)$ quantifies the amount of measurement noise since the measured distances $y_i$ lie in the interval $\big[(1-\alpha)||x^*-r_i||,(1+\alpha)||x^*-r_i||\big]$ which is increasing in size with $\alpha$. 

To compare both methods we generate $100$ localization instances as follows:
\begin{itemize}
	\item We generate $100$ positions for the target $x^*$ and anchors $(r_1,r_2,r_3)$  by sampling uniformly over $[ -1, 1]^2$.
	\item We generate a grid of thirty regularly spaced $\alpha$ points in the interval $[0.05,0.95]$. For each target $x^*$, anchors $(r_1,r_2,r_3)$ and noise level $\alpha$, a random positive definite matrix $\Sigma$ is generated and scaled (see Figure~\ref{figure:experimental_setup_plot}) in order
	to ensure that~\eqref{eqn:noise_level_ellip} holds, that is, any uncertainty vector $u$ changes the true distances $||x^*-r_i||$ by at most $\alpha\%$.
					\begin{figure}[h]
		\begin{center}
			\includegraphics[width=7.5cm]{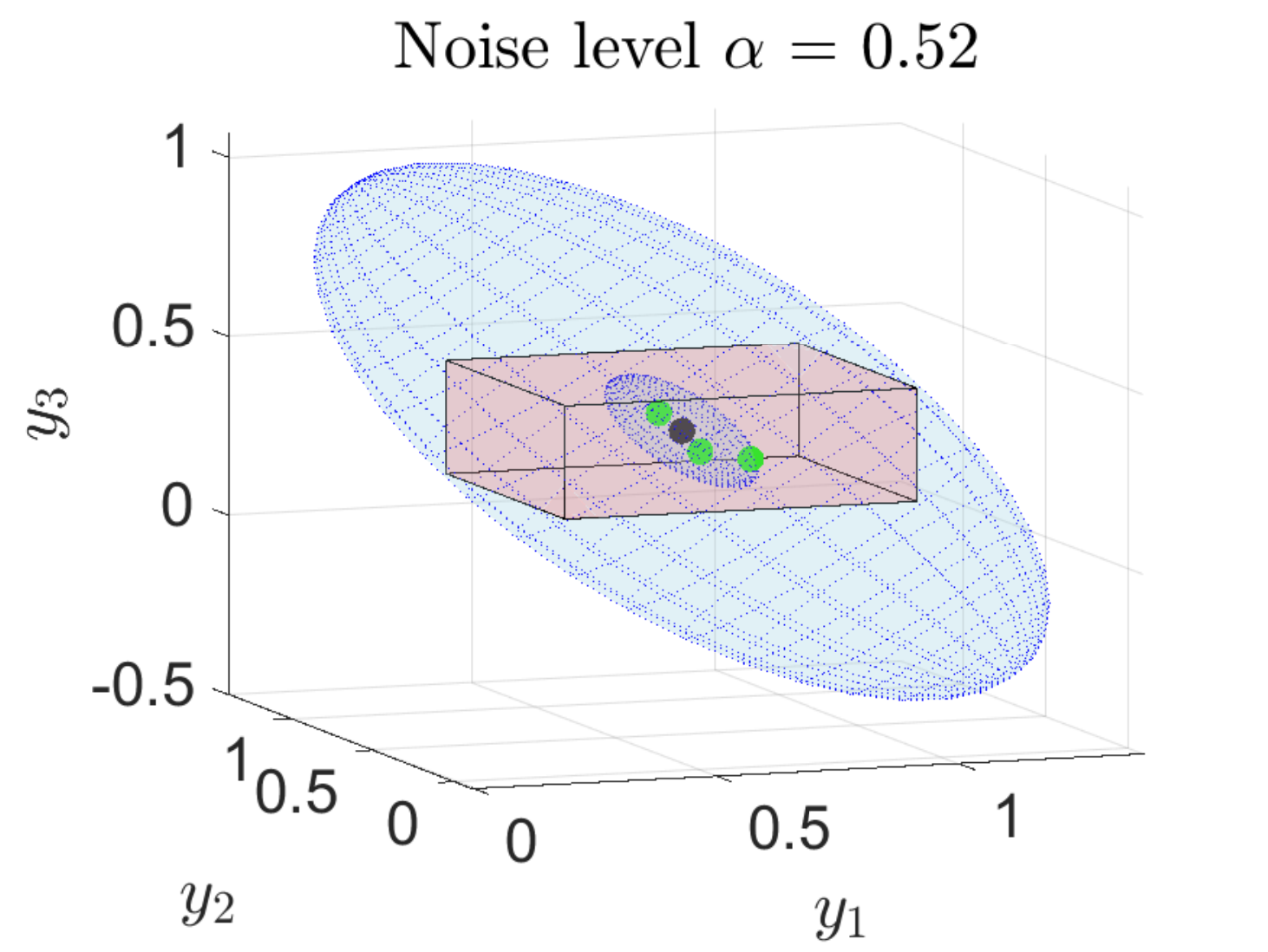}
			\caption{ Generating a typical problem instance: the black dot represents the vector $\theta(x^*;R)$ of true distance measurements; the red rectangle is the set $\theta(x^*;R) \Big[ 1-\alpha , 1+\alpha \Big]$; the largest blue ellipse $\mathcal{E}(\theta(x^*;R) ,\hat{\Sigma})$ is obtained by randomly generating a positive definite $\hat{\Sigma}$; the smallest blue ellipse $\mathcal{E}(\theta(x^*;R) ,{\Sigma})$ is obtained by scaling $\hat{\Sigma}$ by $\lambda=\alpha^2 \min_i\{  \theta(x^*,R)_i^2/\hat{\Sigma}_{i,i}\}$ so $\Sigma:=\lambda \hat{\Sigma}$; the three green dots represent the measurements  $y^{(1)},y^{(2)},y^{(3)}$ generated according to model~\eqref{eqn:measurement_generation}. 	}
			\label{figure:experimental_setup_plot}
		\end{center}
	\end{figure}
	\item Given an uncertainty region $\mathcal{E}(0,\Sigma)$  and anchors $\{a_m\}_m$ we generate three measurements $y^{(1)},y^{(2)},y^{(3)}$
	by adding a random perturbation $u^{(1)},u^{(2)},u^{(3)} \in \mathcal{E}(0,\Sigma)$ to the true measurements $||a_m-x^*||$, that is,
	\begin{align}
	 y^{(l)}:= \theta(x^*;R)+u^{(l)},\,\, u^{(l)} \in \mathcal{E}(0,\Sigma).
	\label{eqn:measurement_generation}
	\end{align}
	We sample three measurements  $y_m^{(1)},y_m^{(2)},y_m^{(3)}$ instead of one such that our experiments are more reliable, that is, our setup  accommodates scenarios where the measurements can be sampled from different regions of the ellipsoidal uncertainty region $ \mathcal{E}(0,\Sigma)$, see Figure~\ref{figure:experimental_setup_plot}. 
\end{itemize} 

\vspace*{1ex}

\noindent\textbf{\textcolor{black}{Grid Approximation of $\mathcal{X}$.}}
Consider the set of measurements~\eqref{eq:model}. One naive option to solve problem~\eqref{eq:core} is simply to perform a grid search over $(x,u)$ pairs: given a measurement vector $y\in \mathbf{R}^M$ generate possible target positions $x$ and define $u$ as 
\begin{equation}
u_m=y_m-||r_m+x||.
\end{equation}
\textcolor{black}{If  $u \in \mathcal{E}(0,\Sigma)$ and $||x-r_m||^2 \geq \Sigma_{m,m}$, then $x\in \mathcal{X}$. This grid method can be useful to approximate $\mathcal{X}$ by a finite set, which we denote by $\mathcal{X}_F$ (the $F$ in $\mathcal{X}_F$ stands for finite). In our context, set $\mathcal{X}_F$ is useful to compare both methods as it serves as a proxy for the true set $\mathcal{X}$. To construct the finite set $\mathcal{X}_F$ we need an initial over-estimator of $\mathcal{X}$, say $\overline {\mathcal X}$, to define a finite grid of $x$ points in $\overline {\mathcal X}_{\text{LFR}}$. After computing rectangles $\overline {\mathcal X}_{\text{LFR}}$ and $\overline {\mathcal X}_{\text{SDP}}$  we create a grid of $400^2$ linearly spaced points $x:=(x_1,x_2)$ that spans the rectangle  $\overline {\mathcal X}_{\text{SDP}} \bigcup\overline {\mathcal X}_{\text{LFR}} $. Set $\mathcal{X}_F$  is simply the point cloud 
	of target positions for which the aforementioned naive method returned a valid uncertainty vector $u\in \mathcal{E}(0,\Sigma)$ when $x$ varies over $ \overline {\mathcal X}_{\text{SDP}} \bigcup\overline {\mathcal X}_{\text{LFR}} $ and $||x-r_m||^2\geq \Sigma_{m,m}$. }
	\vspace*{1ex}
	
	\noindent\textbf{\textcolor{black}{Performance Metric.}}
\textcolor{black}{	By construction, the true set $\mathcal{X}$ can be over-approximated by either  $\overline {\mathcal X}_{\text{LFR}}$ or  $\overline {\mathcal X}_{\text{SDP}}$. Ideally, we would like for set $\overline {\mathcal X}_{\text{LFR}}$ to be ``smaller'' than $\overline {\mathcal X}_{\text{SDP}}$, in the sense that it occupies a smaller area. Let $|A|$ denote the area (Lebesgue measure) of an arbitrary closed set $A\subseteq \mathbf{R}^M$. In this case the difference $|\overline {\mathcal X}_{\text{SDP}}|-|\overline {\mathcal X}_{\text{LFR}}|$ measures the accuracy gain (in squared meters) of our method with respect to the benchmark standard. Since we are computing rectangular approximations of $\mathcal{X}$ let $R(\mathcal{X}_F)\subseteq \mathbf{R}^2$ denote the tightest rectangle~\footnote{\textcolor{black}{Rectangle $R(\mathcal{X}_F)$ can be easily computed since $\mathcal{X}_F$ is just a finite set.}} that encloses $\mathcal{X}_F$ -- see Figure~\ref{fig:worst_average_best} for examples of  $\mathcal{X}_F$ and  $R(\mathcal{X}_F)$. Given $R(\mathcal{X}_F)$,  we can measure the accuracy gain of our method with respect to the actual area of the rectangular approximation of $\mathcal{X}_F$.  This is a more reliable metric as it introduces a sense of scale to the difference $|\overline {\mathcal X}_{\text{SDP}}|-|\overline {\mathcal X}_{\text{LFR}}|$. In concrete, we define the gain factor $G\in \mathbf{R}$ as
	\begin{align}
	G:=\frac{|\overline {\mathcal X}_{\text{SDP}}|-|\overline {\mathcal X}_{\text{LFR}}|}{|R(\mathcal{X}_F)|}.
	\label{eqn:gain_factor}
	\end{align}
	Ideally we would like to have a positive gain $G>0$, the larger the better. For example a gain of $G=13$ units says that our rectangle  $\overline {\mathcal X}_{\text{LFR}}$ occupies an area that is smaller than  $\overline {\mathcal X}_{\text{SDP}}$  by $G=13$ units of $R(\mathcal{X}_F)$. So if we approximate $\mathcal{X}_F$ by the rectangle $R(\mathcal{X}_F)$ then the benchmark $\overline {\mathcal X}_{\text{SDP}}$ adds $G=13$ units of $R(\mathcal{X}_F)$ to our rectangle $\overline {\mathcal X}_{\text{LFR}}$ -- right plot of Figure~\ref{fig:worst_average_best} (a). }
\\~\\
\noindent\textbf{Results.} 
\textcolor{black}{Figure~\ref{fig:numerical_results} plots statistics of the gain factor~\eqref{eqn:gain_factor}. The solid blue curve represents the mean value of $G$ while the filled blue region represents the $90\%$ confidence interval.
\begin{figure}[h]
	\centering
	\includegraphics[width=9cm]{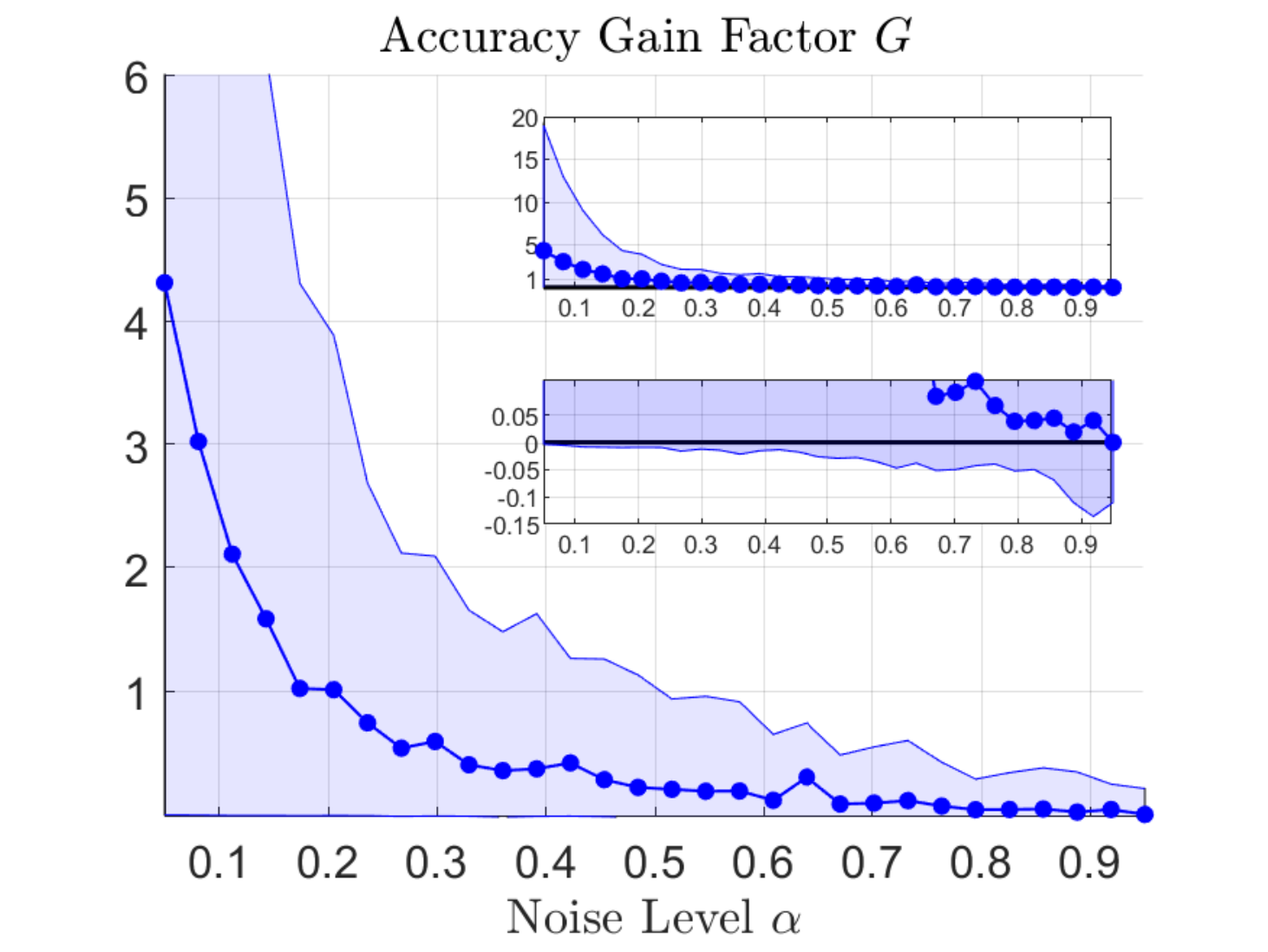}
	\caption{\textcolor{black}{Statistics of the gain factor $G$ as a function of the noise level $\alpha$: how many units of $R(\mathcal{X}_F)$ the benchmark   $\overline {\mathcal X}_{\text{SDP}}$ adds to our relaxation  $\overline {\mathcal X}_{\text{LFR}}$? The solid blue curve with circles represent the average of $G$. The light blue region represents the $90\%$ confidence interval ($5\%$ -- $95\%$ percentiles). The top subplot zooms out on the central figure to fully display the $95\%$ percentile of $G$. The bottom subplot zooms in on the central figure to fully display the  $5\%$ percentile of $G$.}}
	\label{fig:numerical_results}
\end{figure}
\\~\\
For any noise level $\alpha$ there is a benefit in using our LFR relaxation since, in average, the gain factor $G$ is always positive ($G\in(0, 4.31)$). In low noise regimes ($\alpha\leq 0.2 $) the rectangle  $\overline {\mathcal X}_{\text{LFR}}$  is always smaller than  $\overline {\mathcal X}_{\text{SDP}}$ by, at least, one unit of $\mathcal{X}_F$ ($G\in(1, 4.31)$). This means that the benchmark  $\overline {\mathcal X}_{\text{SDP}}$ delivers a loose approximation, in the sense that the rectangle $R(\mathcal{X}_F)$ fits in the extra space induced by $\overline {\mathcal X}_{\text{SDP}}$. For lower values of $\alpha$ the factor $G$ increases rapidly and we can get much higher accuracy gains. For example when $\alpha=0.08$ the rectangle  $\overline {\mathcal X}_{\text{LFR}}$  is much smaller than  $\overline {\mathcal X}_{\text{SDP}}$ since the slack $\overline {\mathcal X}_{\text{LFR}}\setminus \overline {\mathcal X}_{\text{SDP}}$ encompasses $G\approx 3$ units of $R(\mathcal{X}_F)$ -- see the central plot of Figure~\ref{fig:worst_average_best} (a). If we average the mean gain curve of  $G$ (solid blue line) for $\alpha\leq 0.2$, we find that our method delivers a rectangle $\overline {\mathcal X}_{\text{LFR}}$ which approximately doubles the accuracy of $\overline {\mathcal X}_{\text{SDP}}$, that is, $G\approx 2$.
\\~\\ 
 For moderate noise levels ($\alpha\in(0.2,0.67)$) the gain factor $G$ starts to slowly decay, which means that the extra accuracy delivered by $\overline {\mathcal X}_{\text{LFR}}$ tends to be less significant. For example for $\alpha=0.42$ our method now improves the benchmark by less than half a unit of $R(\mathcal{X}_F)$ ($G\approx 0.42$) -- see the central plot of Figure~\ref{fig:worst_average_best} (b). In high-noise regimes  ($\alpha\in(0.67,1)$) the gain factor $G$ tends to oscillate around the value $G\approx 0.05$ (see the bottom alternative axis in Figure~\ref{fig:numerical_results}), so the extra benefit of our method is minimal. These findings suggest that, for an high-noise level $\alpha$, the performance of both relaxations tends to be similar as can be confirmed in the central plot of Figure~\ref{fig:worst_average_best} (c). This is intuitive since the localization problem becomes harder in the sense that, for an increasing amount of measurement noise ($\alpha$ large), the area covered by $\mathcal{X}_F$ tends to increase -- observe how the area of $R(\mathcal{X}_F)$ (black rectangle) increases, when moving down in Figure~\ref{fig:worst_average_best}.
 \\~\\
 The left and right plots of Figure~\ref{fig:numerical_results} also display experiments where the ratio $G$ achieve the lowest ($5\%$) and highest ($95\%$)  percentiles as displayed in two alternatives axis in figure~\ref{fig:numerical_results}. Note that the high $95\%$ percentile can be fairly high, while the $5\%$ percentile is only moderately low. This means that, on the best case, our method can have considerable improvements regarding the benchmark approach  ($G>0$ high) while being only slightly inferior ($G<0$ close to zero) on the worst instances. Consider, for example, the high-noise regime $(\alpha=0.8)$ of  figure 5 (b). The rectangle $\overline {\mathcal X}_{\text{LFR}}$ has a negative worst case gain ($G=-0.05$), which is an order of magnitude smaller than the best case gain ($G=0.3$).
\begin{figure}[h]
	\centering
	\begin{subfigure}{0.5\textwidth}
		\centering
		\includegraphics[width=8.7cm,height=3.7cm]{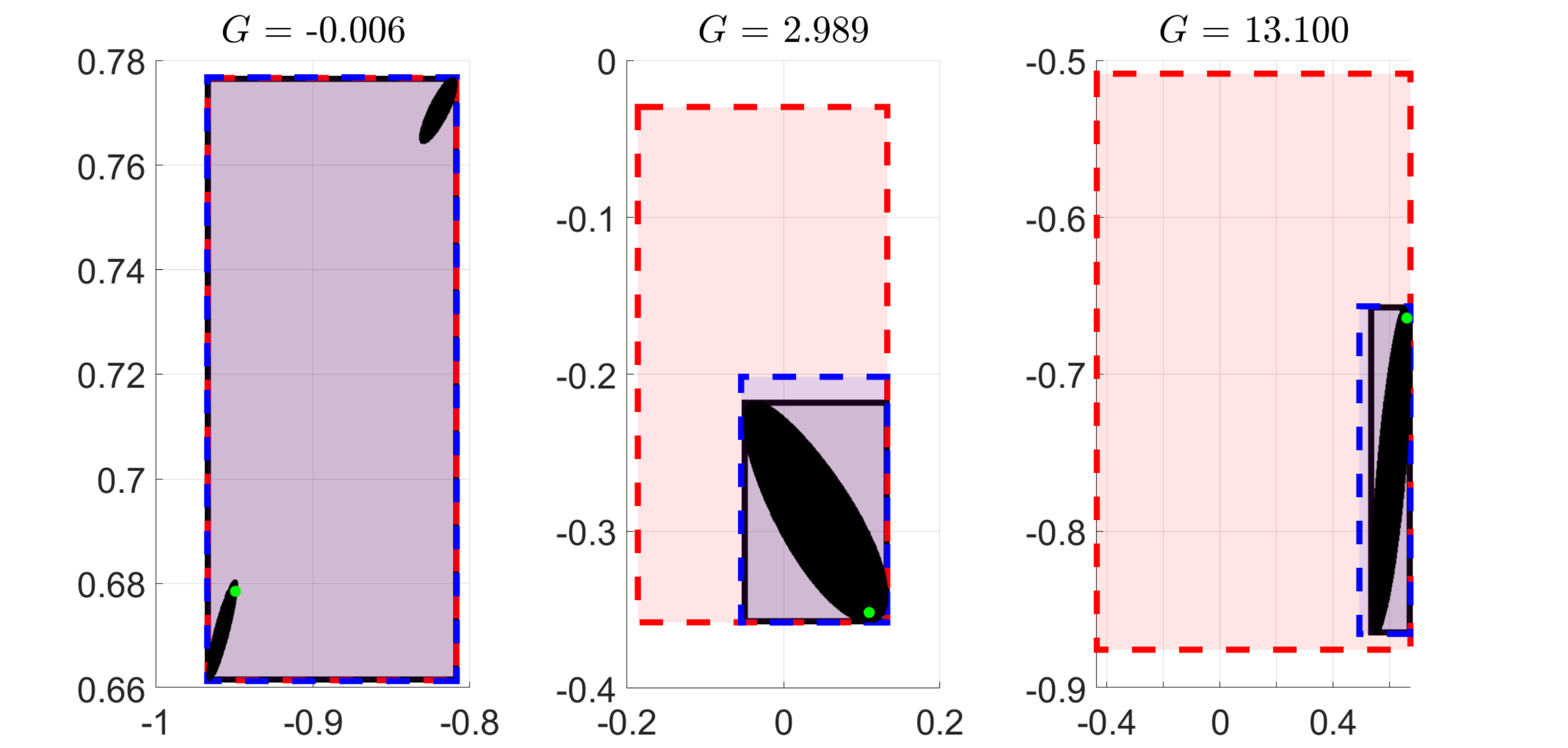}
		\caption{ Low noise $\alpha=0.08$.   }
		\vspace{0.2cm}
		\label{fig_a}
	\end{subfigure}
	\begin{subfigure}{0.5\textwidth}
		\centering
		\includegraphics[width=8.7cm,height=3.7cm]{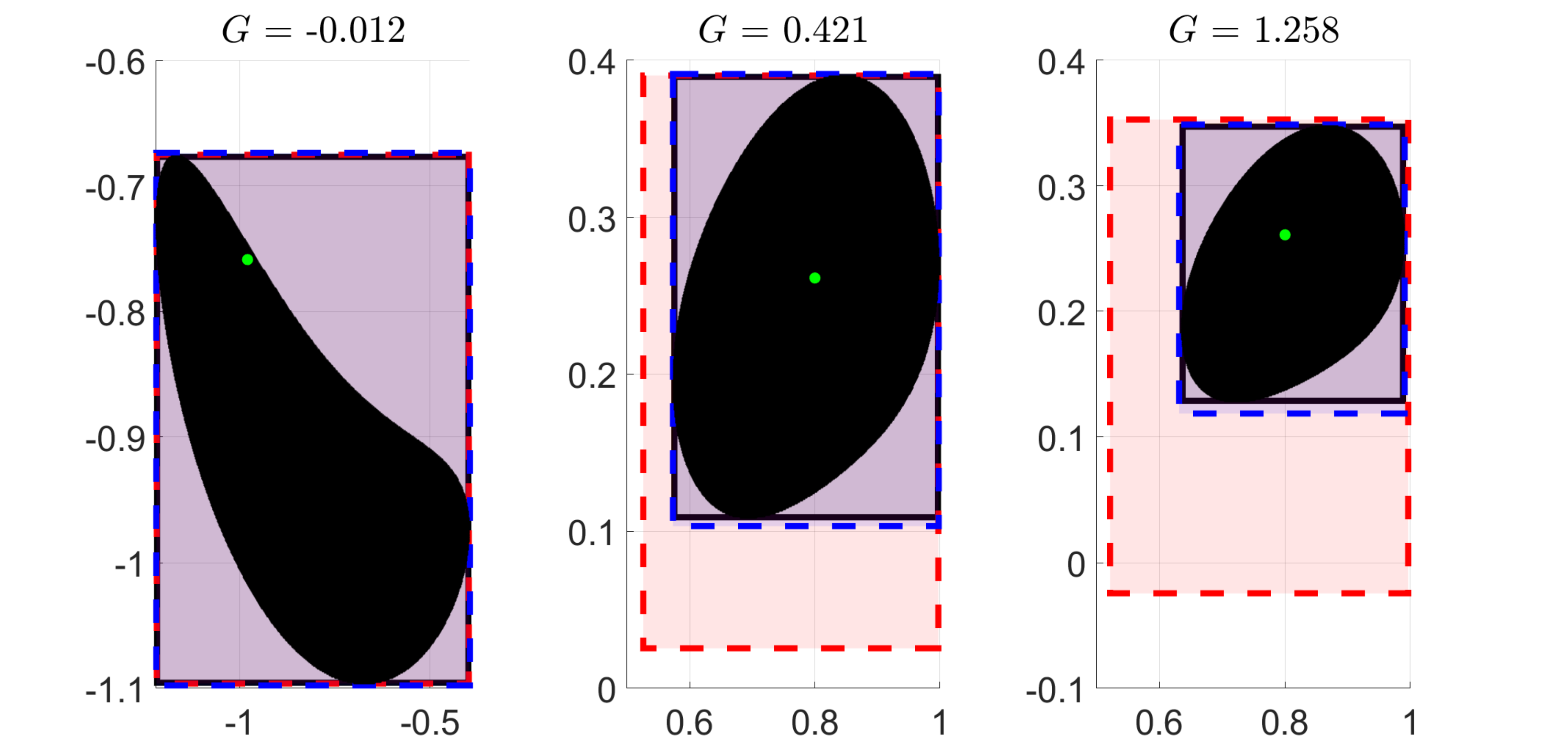}
		\caption{ Moderate noise $\alpha=0.42$.  }
		\vspace{0.2cm}
		\label{fig_b}
	\end{subfigure}  
	\begin{subfigure}{0.5\textwidth}
		\centering
		\includegraphics[width=8.7cm,height=3.7cm]{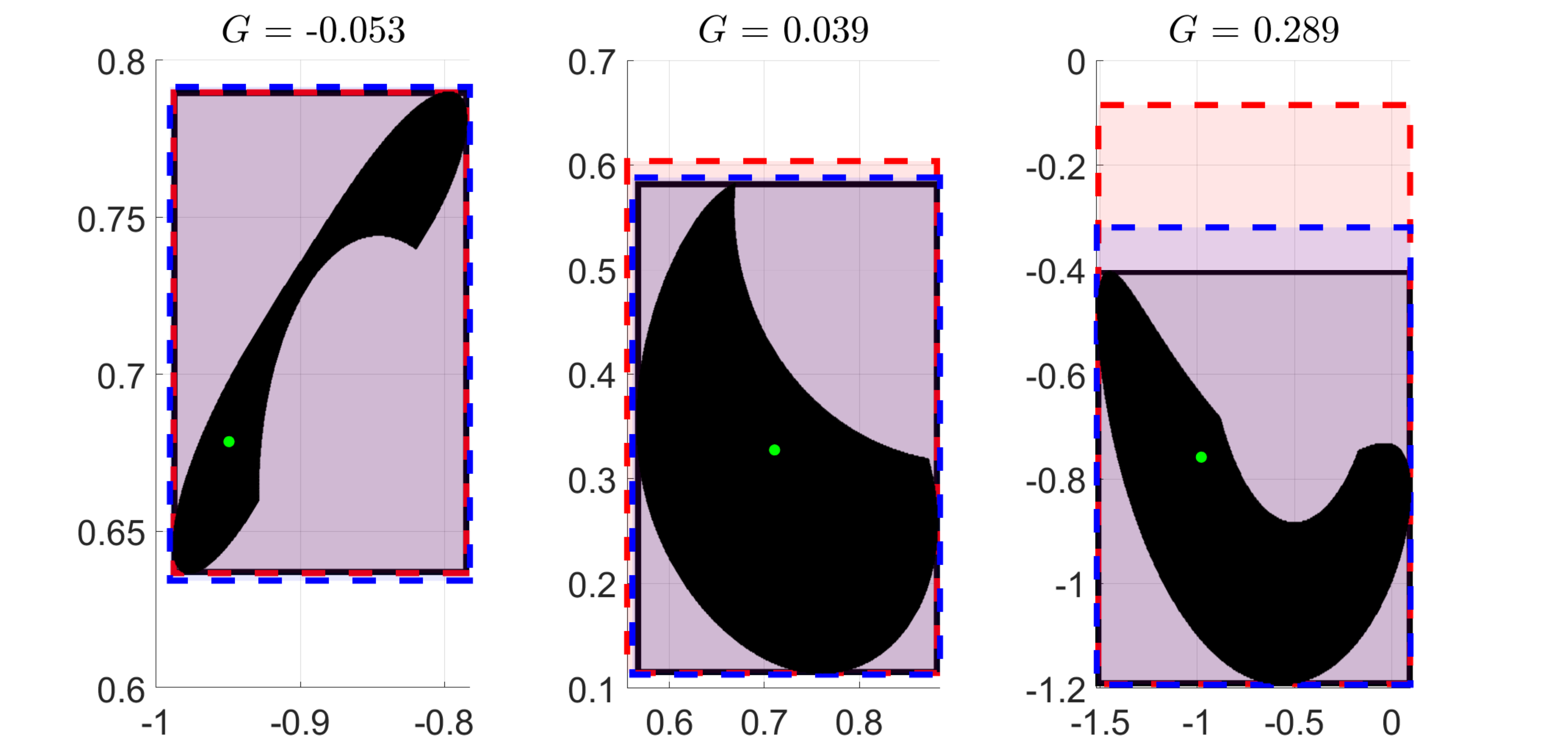}
		\caption{ High noise $\alpha=0.8$.  }
		\label{fig_c}
	\end{subfigure}
	\caption{\textcolor{black}{Experiments with a low, average and high gain $G$ for low, moderate and high values of noise level $\alpha$. We plot the target $x^*$ (green dot), the grid approximation ${\mathcal{X}}_F$ (black region), the tightest rectangle  $R({\mathcal{X}}_F)$ enclosing ${\mathcal{X}}_F$ (black rectangle), the LFR rectangle $\overline {\mathcal X}_{\text{LFR}}$ (blue region) and the SDP rectangle $\overline {\mathcal X}_{\text{SDP}}$  (red region). We want a rectangle $\overline {\mathcal X}_{\text{LFR}}$ that includes ${\mathcal{X}}_F$ while being smaller than the benchmark $\overline {\mathcal X}_{\text{SDP}}$ and not much larger than $R({\mathcal{X}}_F)$. The average gain factor $G$ increases by an order of magnitude from low to moderate and moderate to high values of $\alpha$ -- central plots of all three figures.}	}
	\label{fig:worst_average_best}
\end{figure}
\\~\\
The higher performance of the LFR approach has an associated computational cost. While the SDP approaches takes, on average,  \textcolor{black}{$1.95$sec} to run our method takes about \textcolor{black}{$2.7$sec}. This follows because (1) our method solves an SDP where the dimension of $X$ has a quadratic dependence on the number of anchors $M$ and (2) we impose the flattening inequality, that is, an LMI in $\mathcal{S}^{2M}$. Note that the $X$ matrix coming from the standard method only exhibits a linear dependence on $M$. Furthermore there is no flattening \textcolor{black}{inequality} in~\eqref{eq:coress}. This issue might be secondary for problems with a low number of anchors, but when $M$ is large \textcolor{black}{ it becomes important to study the computational limits of our approach. To approach this issue, figure~\ref{fig:computing_times} plots the computational time (in seconds) of both methods for an increasing number of anchors $M\in\{1,\dots,10\}$.  As seen, our method is clearly slower than the standard relaxation; yet figure~\ref{fig:numerical_results} shows that our method stills runs in reasonable time (say $1$ to $2$ minutes) for a considerable number of anchors (say $M=9$  or $M=10$). Future work includes studying the scalability of our method when the number of anchors $M$ is high.}
\begin{figure}[h]
	\centering
	\includegraphics[width=7cm]{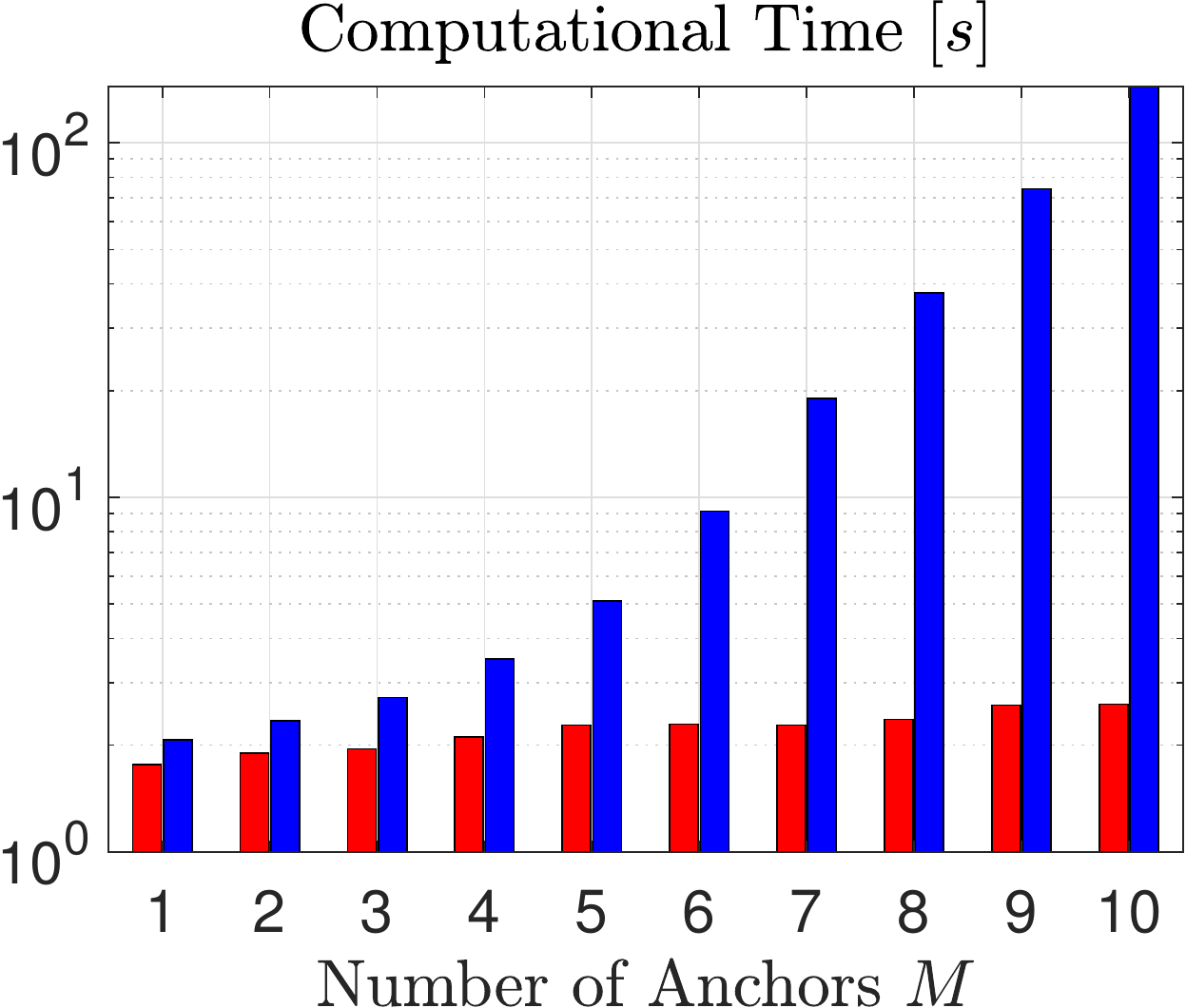}
	\caption{\textcolor{black}{Average computational time of both relaxations for $10$ Monte Carlo trials. The blue bars represent our LFR relaxation and the red bars represent the standard relaxation. Although our method is slower (the blue bars are higher than the red ones), it stills runs in acceptable time (say $1$ to $2$ minutes) for a high number of anchors (say $M=9$  or even $M=10$).  }  }
	\label{fig:computing_times}
\end{figure}
}

\section{\textcolor{black}{Extensions - High Order Approximation of $\mathcal{X}$}}
\label{sec:extensions}
\textcolor{black}{Section~\ref{sec:introduction} considered the problem of outer approximating set $\mathcal{X}$ by a simple rectangular set $\overline{\mathcal{X}}$. The rectangle defined in~\eqref{eq:rect} is a convex polyhedron since it may be written as
	\begin{align}
	\overline{\mathcal{X}}=\Big\{x:   s_l^Tx \leq \beta_l(\Sigma,R,y),\enspace l=1,\dots,2d \Big\}
	\label{eqn:Rectangle_approximation}
	\end{align}
	with $2d$ search directions ($d$ is the ambient dimension where the target lies so $x\in \mathbf{R}^d$) given by
	\begin{align*}
	 \begin{bmatrix} s_1 & \dots & s_d \end{bmatrix}=I_d,\enspace \begin{bmatrix} s_{d+1} & \dots & s_{2d} \end{bmatrix}=-I_d
	\end{align*}
	and extension 	$\beta_l(\Sigma,R,y)$ a upper bound on problem~\eqref{eq:core1} when the search direction $s$ is equal to $s_l$.  The notation $\beta_l(\Sigma,R,y)$ indicates that, in order to upper bound the extension of $\mathcal{X}$ in direction $s_l$, we must use the data of the problem; so, matrix $\Sigma$ that defines the uncertainty region $\mathcal{E}(0,\Sigma)$, the measurements $y$ and the positions of the reference landmarks $R=(r_1,\dots,r_M)$.
\\~\\
We can generalize the rectangle in~\eqref{eq:rect} by creating a outer-approximation $\overline{\mathcal{X}}$ that bounds $\mathcal{X}$ along \textit{more} search directions $s_l$. In concrete assume we increment the set of search directions to include $T$ additional directions $s_{2d+1},\dots,s_{2d+T}$. Let $\overline{\mathcal{X}}^T$ denote the generalization of~\eqref{eqn:Rectangle_approximation} when including directions $s_{2d+1},\dots,s_{2d+T}$, that is,
	\begin{align}
\overline{\mathcal{X}}^T=\Big\{x:   s_l^Tx \leq \beta_l(\Sigma,R,y),\enspace l=1,\dots,2d+T \Big\}.
\label{eqn:Rectangle_approximation_2}
\end{align}
By construction, the sequence of polyhedra $\{\overline{\mathcal{X}}^T\}_{T\geq 1}$ forms a decreasing sequence of outer-estimators of $\mathcal{X}$, that is
\begin{equation}
\forall\,T\geq 1:\enspace {\mathcal{X}} \subseteq \overline{ \mathcal{X}}^{T} \subseteq \overline{ \mathcal{X}}^{T-1} \subseteq \dots \subseteq \overline{ \mathcal{X}}^{1} \subseteq  \overline{\mathcal{X}}^{0}=\overline{ \mathcal{X}}. 
\label{eqn:sequence}
\end{equation}
In words, result~\eqref{eqn:sequence} says that using $T$ additional search directions can only provide a tighter approximation of set $\mathcal{X}$. Figure~\ref{fig:perfomance_computation} plots a numerical example that highlights the performance gains of a tighter approximation $\overline{ \mathcal{X}}^T$.
\begin{figure}[h!]
	\centering
	\includegraphics[width=6cm]{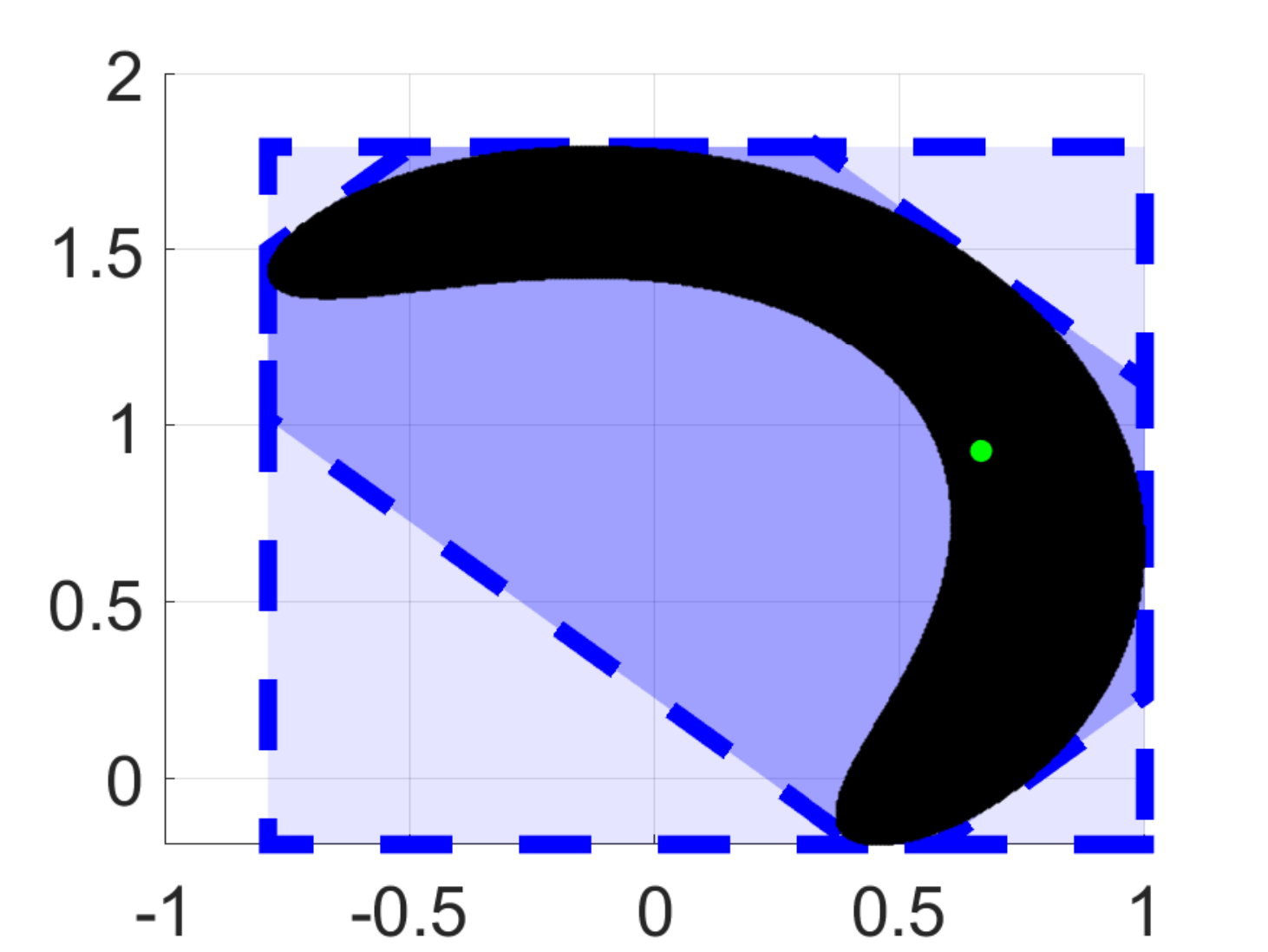}
	\caption{\textcolor{black}{{ Higher order approximation of $\mathcal{X}$ for $T=4$. We plot the target $x^*$ (green dot), the grid approximation ${\mathcal{X}}_F$ (black region), the LFR rectangle $\overline {\mathcal X}_{\text{LFR}}$ (light blue region) and the LFR polyhedron $\overline {\mathcal X}_{\text{LFR}}^{T}$ (dark blue region). The polyhedron $\overline {\mathcal X}_{\text{LFR}}^{T}$  yields a tigher approximation on  ${\mathcal{X}}_F$ by improving the rectangle  $\overline {\mathcal X}_{\text{LFR}}$ along the diagonals $s\in\{(\pm \sqrt{2}/2,\pm \sqrt{2}/2)\}$.	
}}}
	\label{fig:perfomance_computation}
\end{figure}
Result~\eqref{eqn:sequence} implies a clear compromise between performance and computation: when $T$ increases we have a tighter approximation $\overline{ \mathcal{X}}^{T} $ of $\mathcal{X}$ at the expense of computing $T$ additional upper bounds on problem~\eqref{eq:core1}. Future work includes studying the performance-computation trade off of a high-order approximation $\overline{\mathcal{X}}^T$.
}
\section{Conclusion}
\label{sec:conclusion}
This paper considers a different approach for target localization: instead of assuming a fixed noise density and searching for a single point estimate, we track the set of \textit{all} target positions that are consistent with the data model and with two mild assumptions related to the non-negativeness of the measurements and the boundedness of the additive noise. It turns out that this approach is equivalent to tracking the set of ML estimators parametrized by different (unknown) noise densities. Our approach to bound the set of possible targets is to design a polyhedral outer approximation, which is obtained by relaxing a non-convex quadratic program. Our relaxation uses Linear Fractional Representations to model and re-parametrize  the uncertainty vector in the additive data model. Numerical experiments with a rectangular approximation and moderate noise, show that our relaxation tends to outperform a standard SDP relaxation.
\bibliographystyle{IEEEtran}

\bibliography{IEEEabrv,paperLFR}

\begin{thebibliography}{10}
\providecommand{\url}[1]{#1}
\csname url@samestyle\endcsname
\providecommand{\newblock}{\relax}
\providecommand{\bibinfo}[2]{#2}
\providecommand{\BIBentrySTDinterwordspacing}{\spaceskip=0pt\relax}
\providecommand{\BIBentryALTinterwordstretchfactor}{4}
\providecommand{\BIBentryALTinterwordspacing}{\spaceskip=\fontdimen2\font plus
\BIBentryALTinterwordstretchfactor\fontdimen3\font minus
  \fontdimen4\font\relax}
\providecommand{\BIBforeignlanguage}[2]{{%
\expandafter\ifx\csname l@#1\endcsname\relax
\typeout{** WARNING: IEEEtran.bst: No hyphenation pattern has been}%
\typeout{** loaded for the language `#1'. Using the pattern for}%
\typeout{** the default language instead.}%
\else
\language=\csname l@#1\endcsname
\fi
#2}}
\providecommand{\BIBdecl}{\relax}
\BIBdecl

\bibitem{beck2008exact}
A.~Beck, P.~Stoica, and J.~Li, ``Exact and approximate solutions of source
  localization problems,'' \emph{IEEE Transactions on Signal Processing},
  vol.~56, no.~5, pp. 1770--1778, 2008.

\bibitem{yang2009efficient}
K.~Yang, G.~Wang, and Z.-Q. Luo, ``Efficient convex relaxation methods for
  robust target localization by a sensor network using time differences of
  arrivals,'' \emph{IEEE Transactions on Signal Processing}, vol.~57, no.~7,
  pp. 2775--2784, 2009.

\bibitem{ouguz2010convex}
P.~O{\u{g}}uz-Ekim, J.~Gomes, J.~Xavier, and P.~Oliveira, ``A convex relaxation
  for approximate maximum-likelihood 2d source localization from range
  measurements,'' in \emph{2010 IEEE International Conference on Acoustics,
  Speech and Signal Processing}.\hskip 1em plus 0.5em minus 0.4em\relax IEEE,
  2010, pp. 2698--2701.

\bibitem{5714759}
E.~Xu, Z.~Ding, and S.~Dasgupta, ``Source localization in wireless sensor
  networks from signal time-of-arrival measurements,'' \emph{IEEE Transactions
  on Signal Processing}, vol.~59, no.~6, pp. 2887--2897, 2011.

\bibitem{6731596}
H.~Shen, Z.~Ding, S.~Dasgupta, and C.~Zhao, ``Multiple source localization in
  wireless sensor networks based on time of arrival measurement,'' \emph{IEEE
  Transactions on Signal Processing}, vol.~62, no.~8, pp. 1938--1949, 2014.

\bibitem{7747483}
Y.~Wang and K.~C. Ho, ``{TDOA} positioning irrespective of source range,''
  \emph{IEEE Transactions on Signal Processing}, vol.~65, no.~6, pp.
  1447--1460, 2017.

\bibitem{8521706}
Y.~Sun, K.~C. Ho, and Q.~Wan, ``Solution and analysis of {TDOA} localization of
  a near or distant source in closed form,'' \emph{IEEE Transactions on Signal
  Processing}, vol.~67, no.~2, pp. 320--335, 2019.

\bibitem{7926459}
X.~Qu, L.~Xie, and W.~Tan, ``Iterative constrained weighted least squares
  source localization using {TDOA} and {FDOA} measurements,'' \emph{IEEE
  Transactions on Signal Processing}, vol.~65, no.~15, pp. 3990--4003, 2017.

\bibitem{8417931}
X.~Shi and J.~Wu, ``To hide private position information in localization using
  time difference of arrival,'' \emph{IEEE Transactions on Signal Processing},
  vol.~66, no.~18, pp. 4946--4956, 2018.

\bibitem{9089013}
L.~Kraljević, M.~Russo, M.~Stella, and M.~Sikora, ``Free-field tdoa-aoa sound
  source localization using three soundfield microphones,'' \emph{IEEE Access},
  vol.~8, pp. 87\,749--87\,761, 2020.

\bibitem{gao2020majorization}
K.~Gao, J.~Zhu, and Z.~Xu, ``Majorization--minimization-based target
  localization problem from range measurements,'' \emph{IEEE Communications
  Letters}, vol.~24, no.~3, pp. 558--562, 2020.

\bibitem{IhlerFisherMosesWillsky2005}
A.~Ihler, I.~Fisher, J.W., R.~Moses, and A.~Willsky, ``Nonparametric belief
  propagation for self-localization of sensor networks,'' \emph{Selected Areas
  in Communications, IEEE Journal on}, vol.~23, no.~4, pp. 809 -- 819, Apr.
  2005.

\bibitem{vaghefi2015cooperative}
R.~M. Vaghefi and R.~M. Buehrer, ``Cooperative localization in {NLOS}
  environments using semidefinite programming,'' \emph{IEEE Communications
  Letters}, vol.~19, no.~8, pp. 1382--1385, 2015.

\bibitem{ForeroGiannakis2012}
``Sparsity-exploiting robust multidimensional scaling,'' \emph{IEEE
  Transactions on Signal Processing}, vol.~60, no.~8, pp. 4118 --4134, Aug.
  2012.

\bibitem{KorkmazVeen2009}
S.~Korkmaz and A.-J. van~der Veen, ``Robust localization in sensor networks
  with iterative majorization techniques,'' in \emph{Acoustics, Speech and
  Signal Processing, 2009. ICASSP 2009. IEEE International Conference on}, Apr.
  2009, pp. 2049 --2052.

\bibitem{soares2021strong}
C.~Soares and J.~Gomes, ``{STRONG}: Synchronous and asynchronous robust network
  localization, under non-{Gaussian} noise,'' \emph{Signal Processing}, vol.
  185, p. 108066, 2021.

\bibitem{chakraborty2020cooperative}
A.~Chakraborty, K.~M. Brink, and R.~Sharma, ``Cooperative relative localization
  using range measurements without a priori information,'' \emph{IEEE Access},
  vol.~8, pp. 205\,669--205\,684, 2020.

\bibitem{Sun2004}
{Guo-Lin Sun} and {Wei Guo}, ``Bootstrapping {M}-estimators for reducing errors
  due to non-line-of-sight {(NLOS)} propagation,'' \emph{IEEE Communications
  Letters}, vol.~8, no.~8, pp. 509--510, 2004.

\bibitem{YinZoubirFritscheGustafsson2013}
F.~Yin, A.~Zoubir, C.~Fritsche, and F.~Gustafsson, ``Robust cooperative sensor
  network localization via the {EM} criterion in {LOS/NLOS} environments,'' in
  \emph{Signal Processing Advances in Wireless Communications (SPAWC), 2013
  IEEE 14th Workshop on}, June 2013, pp. 505--509.

\bibitem{wang2014nlos}
G.~Wang, H.~Chen, Y.~Li, and N.~Ansari, ``{NLOS} error mitigation for
  {TOA}-based localization via convex relaxation,'' \emph{IEEE Transactions on
  Wireless Communications}, vol.~13, no.~8, pp. 4119--4131, 2014.

\bibitem{tomic2017robust}
S.~Tomic, M.~Beko, R.~Dinis, and P.~Montezuma, ``A robust bisection-based
  estimator for {TOA}-based target localization in {NLOS} environments,''
  \emph{IEEE Communications Letters}, vol.~21, no.~11, pp. 2488--2491, 2017.

\bibitem{Chen2019}
S.~{Chen}, J.~{Zhang}, and C.~{Xu}, ``Robust distributed cooperative
  localization with {NLOS} mitigation based on multiplicative convex model,''
  \emph{IEEE Access}, vol.~7, pp. 112\,907--112\,920, 2019.

\bibitem{7426865}
G.~Wang, A.~M.-C. So, and Y.~Li, ``Robust convex approximation methods for
  {TDOA}-based localization under {NLOS} conditions,'' \emph{IEEE Transactions
  on Signal Processing}, vol.~64, no.~13, pp. 3281--3296, 2016.

\bibitem{9257378}
Y.~Yan, G.~Yang, H.~Wang, and X.~Shen, ``Semidefinite relaxation for source
  localization with quantized {ToA} measurements and transmission uncertainty
  in sensor networks,'' \emph{IEEE Transactions on Communications}, vol.~69,
  no.~2, pp. 1201--1213, 2021.

\bibitem{guvenc2009survey}
I.~Guvenc and C.-C. Chong, ``A survey on {TOA} based wireless localization and
  {NLOS} mitigation techniques,'' \emph{IEEE Communications Surveys \&
  Tutorials}, vol.~11, no.~3, pp. 107--124, 2009.

\bibitem{prorok2012online}
A.~Prorok, L.~Gonon, and A.~Martinoli, ``Online model estimation of
  ultra-wideband {TDOA} measurements for mobile robot localization,'' in
  \emph{2012 IEEE International Conference on Robotics and Automation}.\hskip
  1em plus 0.5em minus 0.4em\relax Ieee, 2012, pp. 807--814.

\bibitem{cong2005nonline}
L.~Cong and W.~Zhuang, ``Nonline-of-sight error mitigation in mobile
  location,'' \emph{IEEE Transactions on Wireless Communications}, vol.~4,
  no.~2, pp. 560--573, 2005.

\bibitem{9350568}
C.~He, Y.~Yuan, and B.~Tan, ``Alternating direction method of multipliers for
  {TOA}-based positioning under mixed sparse {LOS/NLOS} environments,''
  \emph{IEEE Access}, vol.~9, pp. 28\,407--28\,412, 2021.

\bibitem{chen2019improved}
H.~Chen, G.~Wang, and N.~Ansari, ``Improved robust {TOA}-based localization via
  {NLOS} balancing parameter estimation,'' \emph{IEEE Transactions on Vehicular
  Technology}, vol.~68, no.~6, pp. 6177--6181, 2019.

\bibitem{marano2010nlos}
S.~Marano, W.~M. Gifford, H.~Wymeersch, and M.~Z. Win, ``{NLOS} identification
  and mitigation for localization based on {UWB} experimental data,''
  \emph{IEEE Journal on selected areas in communications}, vol.~28, no.~7, pp.
  1026--1035, 2010.

\bibitem{Shi2016}
X.~Shi, B.~D.~O. Anderson, G.~Mao, Z.~Yang, J.~Chen, and Z.~Lin, ``Robust
  localization using time difference of arrivals,'' \emph{IEEE Signal
  Processing Letters}, vol.~23, no.~10, pp. 1320--1324, 2016.

\bibitem{eldar2008minimax}
Y.~C. Eldar, A.~Beck, and M.~Teboulle, ``A minimax {Chebyshev} estimator for
  bounded error estimation,'' \emph{IEEE Transactions on Signal Processing},
  vol.~56, no.~4, pp. 1388--1397, 2008.

\bibitem{first_LFR_paper}
C.~Soares and J.~Xavier, ``Locating a target from uncertain data: convex
  supersets based on linear-fractional representations,'' in \emph{IEEE EUROCON
  2019 -18th International Conference on Smart Technologies}, 2019, pp. 1--5.

\bibitem{packard1993control}
A.~Packard and F.~Wu, ``Control of linear fractional transformations,'' in
  \emph{Proceedings of 32nd IEEE Conference on Decision and Control}.\hskip 1em
  plus 0.5em minus 0.4em\relax IEEE, 1993, pp. 1036--1041.

\bibitem{el1996control}
L.~El~Ghaoui and G.~Scorletti, ``Control of rational systems using
  linear-fractional representations and linear matrix inequalities,''
  \emph{Automatica}, vol.~32, no.~9, pp. 1273--1284, 1996.

\bibitem{zhou1998essentials}
K.~Zhou and J.~C. Doyle, \emph{Essentials of robust control}.\hskip 1em plus
  0.5em minus 0.4em\relax Prentice hall Upper Saddle River, NJ, 1998, vol. 104.

\bibitem{calafiore2004ellipsoidal}
G.~Calafiore and L.~El~Ghaoui, ``Ellipsoidal bounds for uncertain linear
  equations and dynamical systems,'' \emph{Automatica}, vol.~40, no.~5, pp.
  773--787, 2004.

\bibitem{billingsley1995probability}
\BIBentryALTinterwordspacing
P.~Billingsley, \emph{Probability and Measure}, ser. Wiley Series in
  Probability and Statistics.\hskip 1em plus 0.5em minus 0.4em\relax Wiley,
  1995. [Online]. Available:
  \url{https://books.google.pt/books?id=z39jQgAACAAJ}
\BIBentrySTDinterwordspacing

\bibitem{huber2004robust}
P.~J. Huber, \emph{Robust statistics}.\hskip 1em plus 0.5em minus 0.4em\relax
  John Wiley \& Sons, 2004, vol. 523.

\bibitem{ZQ_luo}
Z.-q. Luo, W.-k. Ma, A.~M.-c. So, Y.~Ye, and S.~Zhang, ``Semidefinite
  relaxation of quadratic optimization problems,'' \emph{IEEE Signal Processing
  Magazine}, vol.~27, no.~3, pp. 20--34, 2010.

\bibitem{cvx}
I.~CVX~Research, ``{CVX}: Matlab software for disciplined convex programming,
  version 2.0,'' \url{http://cvxr.com/cvx}, Aug. 2012.

\bibitem{gb08}
M.~Grant and S.~Boyd, ``Graph implementations for nonsmooth convex programs,''
  in \emph{Recent Advances in Learning and Control}, ser. Lecture Notes in
  Control and Information Sciences, V.~Blondel, S.~Boyd, and H.~Kimura,
  Eds.\hskip 1em plus 0.5em minus 0.4em\relax Springer-Verlag Limited, 2008,
  pp. 95--110, \url{http://stanford.edu/~boyd/graph_dcp.html}.

\bibitem{RANTZER19967}
\BIBentryALTinterwordspacing
A.~Rantzer, ``On the {Kalman—Yakubovich—Popov} lemma,'' \emph{Systems \&
  Control Letters}, vol.~28, no.~1, pp. 7--10, 1996. [Online]. Available:
  \url{https://www.sciencedirect.com/science/article/pii/0167691195000631}
\BIBentrySTDinterwordspacing

\end{thebibliography}

\appendices
\section{Phrasing the map $\phi_{x,z}$ in~\eqref{eq:defphi} as the LFR in~\eqref{eq:phiLFR}}
\label{ap:step1}
We write the map $\phi_{x,z}$ as an LFR by writing each component of $\phi_{x,z}$ as an elementary LFR and then by stacking terms. In fact, our derivation uses two properties: LFRs are closed under stackings and closed under compositions with linear maps.
\paragraph*{Stacking LFRs} Stacking $K$ LFRs yields an LFR:
\begin{equation}
    \begin{bmatrix} \begin{bmatrix}
  \begin{array}{@{}c|c@{}}
C_1 & d_1 \\
\hline
B_1 & a_1 
\end{array}
\end{bmatrix}_{[p]}(U) \\ \vdots \\ \begin{bmatrix}
  \begin{array}{@{}c|c@{}}
C_K & d_K \\
\hline
B_K & a_K 
\end{array}
\end{bmatrix}_{[p]}(U) \end{bmatrix} =
 \begin{bmatrix}
    \begin{array}{@{}c|c@{}}
  \begin{array}{ccc} C_1 \\ & \ddots \\ & & C_K \end{array} & \begin{array}{c} d_1 \\ \vdots \\ d_K \end{array} \\
  \hline
\begin{array}{ccc} B_1 \\ & \ddots \\ & & B_K \end{array} & \begin{array}{c} a_1 \\ \vdots \\ a_K \end{array} 
    \end{array}
  \end{bmatrix}_{[pK]} \hspace{-0.3cm} (U).
  \label{eq:util1}
\end{equation}
The output LFR is well posed ($I- \text{diag}(C_1,\dots,C_K) (I_{pK} \otimes U)$ invertible for any $U\in \mathcal{U}$  ) if and only if each of the $K$ individual LFRs are well defined ($I-C_k(I_p \otimes U)$ invertible for any $U\in \mathcal{U}$ and index $k$).
When the LFRs being stacked share the same header (same $C_k$ and $d_k$), the formula actually simplifies to
\begin{equation}
    \begin{bmatrix} \begin{bmatrix}
    \begin{array}{@{}c|c@{}}
  C & d \\
  \hline
B_1 & a_1 
    \end{array}
  \end{bmatrix}_{[p]}(U) \\ \vdots \\ \begin{bmatrix}
     \begin{array}{@{}c|c@{}}
  C & d \\
  \hline
B_K & a_K 
    \end{array}
  \end{bmatrix}_{[p]}(U) \end{bmatrix} =
  \begin{bmatrix}
     \begin{array}{@{}c|c@{}}
  C & d \\
  \hline
\begin{array}{c} B_1 \\ \vdots \\ B_K \end{array} & \begin{array}{c} a_1 \\ \vdots \\ a_K \end{array} 
    \end{array}
  \end{bmatrix}_{[p]}(U).
  \label{eq:util2}
\end{equation}
\paragraph*{Composing an LFR with a linear map} Composing an LFR with a linear map $U \mapsto G U H$ yields another LFR,
\begin{equation}
\begin{bmatrix}
\begin{array}{@{}c|c@{}}
C & d \\
\hline
B & a 
\end{array}
\end{bmatrix}_{[p]} \hspace{-0.2cm} (GUH) \hspace{-0.05cm} =\hspace{-0.1cm}
\begin{bmatrix}
\begin{array}{@{}c|c@{}}
  (I_p \otimes H) C (I_p \otimes G) & (I_p \otimes H) d \\
  \hline
B (I_p \otimes G) & a 
    \end{array}
  \end{bmatrix}_{[p]}\hspace{-0.2cm}(  U ).
  \label{eq:util3}
\end{equation}
The verification of~\eqref{eq:util1}, \eqref{eq:util2}, and \eqref{eq:util3} is omitted due to space constraints. Furthermore, for the equality in \eqref{eq:util3}, note that the left-hand side LFR is well defined ($I-C(I_p \otimes \{GUH\})$ invertible for any $U\in \mathcal{U}$ )  if and only if the right-hand side LFR is well defined ( matrix $I-(I_p \otimes H) C (I_p \otimes G)(I_p \otimes U)$ invertible for any $U\in \mathcal{U}$ ).
To phrase map $\phi_{x,z}$ as an LFR, we recognize a stacking structure $$\phi_{x,z}(U) = \begin{bmatrix} \phi_{x,z}^{(1)}(U) \\ \phi_{x,z}^{(2)}(U) \end{bmatrix},\enspace  \phi_{x,z}^{(1)}(U) =  \begin{bmatrix} y_1  - (U e_1)_1 \\  \vdots \\ y_M  - (U e_1)_M \end{bmatrix},$$
where the second auxiliary mapping $\phi_{x,z}^{(2)}(U) $ is given by 
{\small\begin{align}
\phi_{x,z}^{(2)}(U)  \hspace{-0.03cm} = \hspace{-0.1cm} 
\begin{bmatrix} y_1^2  - \left\| r_1 \right\|^2  + 2 r_1^T x  - z - 2 y_1 (U e_1)_1 + (U e_1)_1^2 \\ \vdots \\
y_M^2  - \left\| r_M \right\|^2  + 2 r_M^T x  - z - 2 y_M (U e_1)_M + (U e_1)_M^2  \end{bmatrix} \hspace{-0.15cm}. \nonumber
\end{align}}
 \paragraph*{$\phi_{x,z}^{(1)}$ as an LFR} We start by expressing the generic component $y_m - (U e_1)_m$ of $\phi_{x,z}^{(1)}$ as an elementary LFR
 \begin{equation} y_m - (U e_1)_m  =   
 \begin{bmatrix}
    \begin{array}{c|c}
  \begin{array}{cc} 0 & 1 \\ 0 & 0 \end{array} & \begin{array}{c} 0 \\ 1 \end{array} \\
  \hline
\begin{array}{cc} 0 & -1 \end{array} & y_m 
    \end{array}
  \end{bmatrix}_{[2]}( (U e_1)_m ). \label{eq:phi1}
  \end{equation}
  Now, noting that $(U e_1)_m$ (the $m$th component of the $M$-dimensional vector $U e_1$) can be written as $e_m^T U e_1$ (where $e_m$ is the $m$th column of $I_M$~\eqref{eq:IM}), we can interpret~\eqref{eq:phi1} as an LFR composed with a linear map. Property~\eqref{eq:util3} leads to
  \begin{align}{\tiny
y_m - (U e_1)_m} { \footnotesize =
 \begin{bmatrix} \begin{array}{@{}c|c@{}}
    (I_2 \otimes e_1) \begin{bmatrix} 0 & 1 \\ 0 & 0  \end{bmatrix} (I_2 \otimes e_m^T) & (I_2 \otimes e_1) \begin{bmatrix} 0 \\ 1 \end{bmatrix}  \\
  \hline
\begin{bmatrix} 0 & -1 \end{bmatrix} (I_2 \otimes e_m^T) \rule{0ex}{3ex} & y_m 
    \end{array}
  \end{bmatrix}_{[2]}\hspace{-0.3cm}( U  ).}    \label{eq:phi3}
  \end{align}
  Note that the LFR \eqref{eq:phi1} is always well defined since, for any $U$ matrix, the matrix $I_2-\begin{bmatrix} 0 & 1 \\ 0 & 0  \end{bmatrix} (Ue_1)_m$ is invertible since it is upper triangular with non-zeros on the main diagonal. This also implies that~\eqref{eq:phi3} is well defined for any $U$ due to property~\eqref{eq:util3}. Stacking LFRs~\eqref{eq:phi3} with $1\leq m\leq M$ yields
 \begin{equation}
     \phi_{x,z}^{(1)}(U) = \begin{bmatrix}
    \begin{array}{c|c}
  C & d \\
  \hline
B_1 & a_1
    \end{array}
  \end{bmatrix}_{[2M]}( U), \label{eq:LFRphi1}
 \end{equation}
 with $C$, $d$, $B_1$, and $a_1$ as given in~\eqref{eq:matLFR} and~\eqref{eq:LFRrow1}.

\paragraph*{$\phi_{x,z}^{(2)}$ as an LFR} We start with the elementary LFR
{\begin{align}
y_m^2  &- \left\| r_m \right\|^2  + 2 r_m^T x  - z - 2 y_m (U e_1)_m + (U e_1)_m^2 \nonumber \\  = & 
 \begin{bmatrix}
   \begin{array}{@{}c|c@{}}
  \begin{array}{cc} 0 & 1 \\ 0 & 0 \end{array} & \begin{array}{c} 0 \\ 1 \end{array} \\
  \hline
\begin{array}{cc} 1 & -2 y_m \end{array} & y_m^2  - \left\| r_m \right\|^2  + 2 r_m^T x  - z 
    \end{array}
  \end{bmatrix}_{[2]}\hspace{-0.2cm}( (U e_1)_m ). \nonumber
\end{align}}
 Repeating the steps that led to $\phi_{x,z}^{(1)}$, we use the composition property~\eqref{eq:util3} and then the stacking property~\eqref{eq:util1} to get
 \begin{equation}
     \phi_{x,z}^{(2)}(U) = \begin{bmatrix}
    \begin{array}{c|c}
  C & d \\
  \hline
B_2 & a_2
    \end{array}
  \end{bmatrix}_{[2M]}( U), \label{eq:LFRphi2}
 \end{equation}
 with $C$, $d$, $B_2$, and $a_2$ as given in~\eqref{eq:matLFR} and~\eqref{eq:LFRrow2}. Finally, if we stack~\eqref{eq:LFRphi1} and~\eqref{eq:LFRphi2} using property~\eqref{eq:util2}, we arrive at~\eqref{eq:phiLFR}.
\section{Deriving the flattening map $L_{{\mathcal U}_{[2M]}}$ in~\eqref{eq:thefmap}}
\label{ap:step2}

Beginning at~\eqref{eq:fmap1}, we note 
\begin{align*}
v = (I_{2M} \otimes U) w  &\Leftrightarrow { \begin{bmatrix} v_1 \\ \vdots \\ v_{2M} \end{bmatrix} = \begin{bmatrix} U \\ & \ddots \\ & & U \end{bmatrix}
\begin{bmatrix} w_1 \\ \vdots \\ w_{2M} \end{bmatrix} }\\
&  \Leftrightarrow { \underbrace{\begin{bmatrix} v_1 & \cdots & v_{2M} \end{bmatrix}}_{V} = U \underbrace{\begin{bmatrix} w_1 & \cdots & w_{2M} \end{bmatrix}}_{W} } \\
& \Leftrightarrow  \Sigma^{-1/2} V = Z W,\enspace Z:= \Sigma^{-1/2} U.
\end{align*}
Matrix $U$ belongs to the set ${\mathcal U} = \{ U \colon U^T \Sigma^{-1} U \preceq I \}$, which implies that matrix $Z$ is in the set ${\mathcal Z} = \{ Z \colon Z^T Z \preceq I \}$. Thus, 
\begin{align*}
\exists \,\, U\in \mathcal{U}:\,\,  v = (I_{2M} \otimes U) w &\Leftrightarrow \exists \,\,Z\in \mathcal{Z}: \,\,  \Sigma^{-1/2} V = Z W, \\
 & \Leftrightarrow \, W^T W - V^T \Sigma^{-1} V \succeq 0,
\end{align*}
where the last equivalence uses Lemma 3 (ii) in~\cite{RANTZER19967}. Note that this result generalizes that of example~\eqref{eqn:aux_example}. Finally, 
\begin{align*}
W^T W & = 
\sum_{m = 1}^M E_m w w^T E_m^T,\enspace \enspace V^T \Sigma^{-1} V  =  \sum_{m = 1}^M F_m v v^T F_m^T
\end{align*}
where $E_m = I_{2M} \otimes e_m^T$ and $F_m = I_{2M} \otimes e_m^T \Sigma^{-1/2}$.  In sum, we have shown that $v = (I_{2M} \otimes U) w$ for some $U \in {\mathcal U}$ if and only if $${\mathcal L}_{{\mathcal U}_{[2M]}} \left( \begin{bmatrix} v \\ w \end{bmatrix} \begin{bmatrix} v \\ w \end{bmatrix}^T \right) \succeq 0,$$
where the flattening map ${\mathcal L}_{{\mathcal U}_{[2M]}}$ is as in~\eqref{eq:thefmap}.
\section{Auxiliar matrices for the benchmark method}
\label{ap:aux_matrices}
The  matrices  appearing in reformulation~\eqref{eq:coress} are given by
{\begin{align}
\hat{S}&=\begin{bmatrix}  0 & 0 & 0 & s/2 \\    & 0 & 0 & 0 \\ & &  0 & 0 & \\ & & & 0 \end{bmatrix}, \enspace \hat{G}_m =\begin{bmatrix} 0 & 0 & 0 & 0 \\ & 0 & 0 & 0 \\  & & 0 & -e_m / 2 \\ & & & y_m \end{bmatrix}\nonumber\\ 
\hat{L}_m&=\begin{bmatrix} 0 & 0 & 0 & -r_m \\ & 0 & 0 & 1/2 \\ &  & 0 & 0 & \\ & & & \left\| r_m \right\|^2 - \Sigma_{mm}\end{bmatrix}, \enspace \hat{f}=\begin{bmatrix}  0 \\ 0 \\ 0 \\ 1 \end{bmatrix}  \nonumber\\ 
\hat{K} &=\begin{bmatrix} I_M & 0 & 0 & 0 \\ & 0 & 0 & -1/2 \\ & & 0 & 0 \\  & & & 0  \end{bmatrix}, \enspace  \hat{J} =\begin{bmatrix} 0 & 0 & 0 & 0 \\ & 0 & 0 & 0 \\ & & -\Sigma^{-1} & 0 \\ & & & 1 \end{bmatrix} \nonumber \\
\hat{H}_m&=\begin{bmatrix} 0 & 0 & 0 & r_m \\ & 0 & 0 & -1/2 \\ & & e_m e_m^T & -y_m e_m  \\ & & & y_m^2 - \left\| r_m \right\|^2  \end{bmatrix},
\end{align}} 
with $e_m$ the $m$-th column of the identity matrix $I_M$~\eqref{eq:IM}.

\end{document}